\documentclass[10pt,reqno]{amsart} 
\usepackage{amssymb}
\usepackage{mathtools} 
\usepackage{mathabx} 
\usepackage{thmtools}
\usepackage[backref=page,colorlinks=true]{hyperref} 
\usepackage[font=footnotesize,margin={8pt,8pt}]{caption}
\usepackage{tikz} 
\usetikzlibrary{patterns,decorations.markings,decorations.text}
\usepackage{tikz-cd} 

\hypersetup{
	colorlinks=false,  
	linkbordercolor={.95 .8 .8}, 
	citebordercolor={.8 .95 .8},
	urlbordercolor={.8 .8 .95},
}

\makeatletter

\def\myMRbibitem{\@ifnextchar[\my@lbibitem\my@bibitem}

\def\mybiblabel#1#2{\@biblabel{{\hyperref{http://www.ams.org/mathscinet-getitem?mr=#1}{}{}{#2}}}}

\def\myhyperanchor#1{\Hy@raisedlink{\hyper@anchorstart{cite.#1}\hyper@anchorend}}

\def\my@lbibitem[#1]#2#3#4\par{%
    \item[\mybiblabel{#2}{#1}\myhyperanchor{#3}\hfill]#4%
    \@ifundefined{ifbackrefparscan}{}{\BR@backref{#3}}%
    \if@filesw{\let\protect\noexpand\immediate
       \write\@auxout{\string\bibcite{#3}{#1}}}\fi\ignorespaces%
}

\def\my@bibitem#1#2#3\par{%
    \refstepcounter\@listctr
    \item[\mybiblabel{#1}{\the\value\@listctr}\myhyperanchor{#2}\hfill]#3%
    \@ifundefined{ifbackrefparscan}{}{\BR@backref{#2}}%
    \if@filesw\immediate\write\@auxout
        {\string\bibcite{#2}{\the\value\@listctr}}\fi\ignorespaces%
}

\makeatother

\renewcommand*{\backref}[1]{}
\renewcommand*{\backrefalt}[4]{\quad \tiny
    \ifcase #1 (NOT CITED.)%
    \or        (Cited on page~#2.)%
    \else      (Cited on pages~#2.)%
    \fi}

\declaretheoremstyle[
headfont=\normalfont\itshape,
bodyfont=\normalfont,
qed=\ensuremath{\triangleleft}
]{myremark}

\declaretheorem[name=Theorem]{maintheorem}
\declaretheorem[numberwithin=section]{theorem}
\declaretheorem[sibling=theorem]{lemma}
\declaretheorem[sibling=theorem]{corollary}
\declaretheorem[sibling=theorem]{proposition}
\declaretheorem[sibling=theorem, style=remark]{claim}
\declaretheorem[sibling=theorem, style=myremark]{remark}

\declaretheorem[numbered=no, style=remark, name=Acknowledgements]{ack}

\DeclareFontFamily{U} {MnSymbolA}{}
\DeclareFontShape{U}{MnSymbolA}{m}{n}{
   <-6> MnSymbolA5
   <6-7> MnSymbolA6
   <7-8> MnSymbolA7
   <8-9> MnSymbolA8
   <9-10> MnSymbolA9
   <10-12> MnSymbolA10
   <12-> MnSymbolA12}{}
\DeclareFontShape{U}{MnSymbolA}{b}{n}{
   <-6> MnSymbolA-Bold5
   <6-7> MnSymbolA-Bold6
   <7-8> MnSymbolA-Bold7
   <8-9> MnSymbolA-Bold8
   <9-10> MnSymbolA-Bold9
   <10-12> MnSymbolA-Bold10
   <12-> MnSymbolA-Bold12}{}
\DeclareSymbolFont{MnSyA} {U} {MnSymbolA}{m}{n}

\DeclareMathSymbol{\top}{\mathord}{MnSyA}{219}
\DeclareMathSymbol{\bot}{\mathord}{MnSyA}{217}

\hypersetup{bookmarksdepth = 3} 
\numberwithin{equation}{section}         

\newcommand{\C}{\mathbb{C}}
\newcommand{\R}{\mathbb{R}}

\newcommand{\Z}{\mathbb{Z}}

\renewcommand{\P}{\mathbb{P}}
\newcommand{\B}{\mathbb{B}}
\newcommand{\K}{\mathbb{K}}
\newcommand{\D}{\mathbb{D}}

\newcommand{\cC}{\mathcal{C}}
\newcommand{\cE}{\mathcal{E}}
\newcommand{\cH}{\mathcal{H}}\newcommand{\cI}{\mathcal{I}}

\newcommand{\cM}{\mathcal{M}}\newcommand{\cN}{\mathcal{N}}

\newcommand{\cU}{\mathcal{U}}
\newcommand{\cV}{\mathcal{V}}

\newcommand{\GL}{\mathrm{GL}}
\newcommand{\SL}{\mathrm{SL}}

\newcommand{\sA}{\mathsf{A}}
\newcommand{\fm}{\mathfrak{m}}

\DeclareMathOperator{\supp}{supp}
\DeclareMathOperator{\Int}{int}
\DeclareMathOperator{\tr}{tr}

\newcommand{\tribar}[1]{\mathopen{| {\kern -1.5pt} | {\kern -1.5pt} |} {#1} \mathclose{| {\kern -1.5pt} | {\kern -1.5pt} |}}

\newcommand{\arxiv}[1]{Preprint \href{http://arxiv.org/abs/#1}{arXiv:{#1}}}

\renewcommand{\epsilon}{\varepsilon}
\renewcommand{\phi}{\varphi}
\renewcommand{\emptyset}{\varnothing}
\renewcommand{\setminus}{\smallsetminus} 
\renewcommand{\angle}{\measuredangle}

\newcommand{\multicone}{M}
\newcommand{\multiconedual}{M_\mathrm{co}}
\newcommand{\project}[1]{#1'}    

\begin{document}

\title[Lyapunov-optimizing measures]{The entropy of Lyapunov-optimizing measures of some matrix cocycles}

\author[J.~Bochi]{Jairo Bochi}
\address{Facultad de Matem\'aticas, Pontificia Universidad Cat\'olica de Chile, Av.~Vicu\~{n}a Mackenna 4860, Santiago, Chile}
\urladdr{\href{http://www.mat.uc.cl/~jairo.bochi}{www.mat.uc.cl/$\sim$jairo.bochi}}
\email{jairo.bochi@mat.uc.cl}

\author[M.~Rams]{Micha{\l} Rams}
\address{Institute of Mathematics, Polish Academy of Sciences, ul.~\'{S}niadeckich~8. 00-956 Warsaw, Poland}
\urladdr{\href{http://www.impan.pl/~rams}{www.impan.pl/$\sim$rams}}
\email{rams@impan.pl}

\date{March, 2016 (this version)}

\begin{abstract}
We consider one-step cocycles of $2 \times 2$ matrices, and we are interested in their Lyapunov-optimizing measures, i.e., invariant probability measures that maximize or minimize a Lyapunov exponent. If the cocycle is dominated, that is, the two Lyapunov exponents are uniformly separated along all orbits, then Lyapunov-optimizing measures always exist and are characterized by their support. Under an additional hypothesis of nonoverlapping between the cones that characterize domination, we prove that the Lyapunov-optimizing measures have zero entropy. This conclusion certainly fails without the domination assumption, even for typical one-step $\mathrm{SL}(2,\mathbb{R})$-cocycles; indeed we show that in the latter case there are measures of positive entropy with zero Lyapunov exponent.
\end{abstract}
\maketitle

\section{Introduction}

\emph{Ergodic Optimization} is concerned with the maximization or minimization
of Birkhoff averages of a given function (called the \emph{potential})
over a given dynamical system: see \cite{Jenkinson}.
A paradigm of this subject is that for sufficiently
hyperbolic base dynamics and for typical potentials, optimizing orbits should have low dynamical complexity.
This is confirmed in by a recent result by Contreras~\cite{Contreras},
who showed that the optimizing orbits with respect to generic Lipschitz potentials
over an expanding base are periodic.
An important component of Contreras' proof is the
fact previously shown by Morris \cite{Morris08} that
in this generic situation, optimizing orbits have subexponential complexity
(i.e., zero entropy).

\medskip

In this paper we are interested in Ergodic Optimization in a noncommutative setting.
We will replace Birkhoff sums by matrix products,
and the quantities we want to maximize or minimize are the associated Lyapunov exponents.
We would like to know whether the low complexity phenomena mentioned above
is also typical in this noncommutative setting.

A natural starting point is to consider \emph{one-step} matrix cocycles.
In this case, the optimization problems above can be restated
in more elementary terms: we are given finitely many square matrices,
and we want to find sequences of products of these matrices attaining
the maximum or minimum growth rate.
These maximization and minimization problems were
first considered by Rota and Strang \cite{RotaStrang} and by Gurvits \cite{Gurvits},
respectively.
The associated growth rates are called \emph{joint spectral radius}
and \emph{joint spectral subradius}, respectively;
they play an important role in Control Theory
and there is a large body of literature about them (especially about the former):
see the monograph \cite{Jungers_book} and references therein.
An important contribution to this field was made
by Bousch and Mairesse \cite{BM} who showed that the maximizing products
are not always periodic, thus disproving the so called
\emph{Finiteness Conjecture}.
(Of course, for one-step cocycles the parameter space is finite-dimensional
and thus perturbative arguments are more difficult.)

\medskip

In this paper we deal with $2 \times 2$ one-step cocycles.
We give explicit open conditions that ensure that
the Lyapunov-optimizing orbits form a set of low complexity,
more precisely of zero topological entropy.
These conditions are related to hyperbolicity on the projective space,
and are satisfied in some of the counterexamples to the Finiteness Conjecture
exhibited in the literature.

In order to substantiate the importance of the hyperbolicity hypotheses,
we also show that for typical non-hyperbolic one-step $\SL(2,\R)$-cocycles
the set of minimizing orbits has positive topological entropy.

Let us proceed with the precise definitions and results.

\subsection{Extremal Lyapunov exponents for $2 \times 2$ matrix cocycles}

Let $\Omega$ be a compact metric space and
let  $T \colon \Omega \to \Omega$ be a continuous transformation.
Let $A \colon \Omega \to \GL(2,\R)$ be a continuous map.
The pair $(T,A)$ is a called a \emph{$2 \times 2$ matrix cocycle}.
We are interested in the following products, the logarithms of which play the role of Birkhoff sums in our noncommmutative setting:
\begin{equation} \label{eq:main_products}
A^{(n)}(\omega) \coloneqq A(T^{n-1} \omega) \cdots A(\omega) , \quad
\omega \in \Omega, \ n \ge 0.
\end{equation}
The \emph{Lyapunov exponents} of the cocycle at a point $\omega \in \Omega$,
when they exist, are the limits:
\begin{equation}\label{eq:LE}
\lambda_1(A, \omega) \coloneqq \lim_{n \to +\infty} \frac{1}{n} \log \| A^{(n)}(\omega) \| \, , \quad
\lambda_2(A, \omega) \coloneqq \lim_{n \to +\infty} \frac{1}{n} \log \fm(A^{(n)}(\omega))  \, .
\end{equation}
where, for definiteness, $\|L\|$ is the Euclidian operator norm of
a matrix $L$ (i.e.\ the largest singular value of $L$), and $\fm(L) \coloneqq \|L^{-1}\|^{-1}$ is its \emph{mininorm} (i.e.\ the smallest singular value of $L$).

Let $i\in\{1,2\}$.
If $\mu$ is a $T$-invariant probability measure then
$\lambda_i(A,\omega)$ exists for $\mu$-almost every $\omega$,
we denote $\lambda_i(A,\mu) = \int \lambda_i(A,\omega) \, d\mu(\omega)$.
If $\mu$ is ergodic then $\lambda_i(A,\omega) = \lambda_i(A,\mu)$ for $\mu$-almost every $\omega$.

The \emph{maximal} (or \emph{top}) and \emph{minimal} (or \emph{bottom})
Lyapunov exponents are defined respectively as:
\begin{equation}\label{eq:extremal_lambdas}
\lambda^\top_i (A) \coloneqq \sup_{\mu \in \cM_T} \lambda_i(A,\mu) \, , \quad
\lambda^\bot_i (A) \coloneqq \inf_{\mu \in \cM_T} \lambda_i(A,\mu) \, ,
\end{equation}
where $\cM_T$ denotes the set of all $T$-invariant Borel probability measures.
These four numbers are called the \emph{extremal Lyapunov exponents} of the cocycle.

A basic question is whether the $\sup$'s and $\inf$'s that appear in
\eqref{eq:extremal_lambdas} are attained.
The answer is ``yes'' in the cases of $\lambda^\top_1$ and $\lambda^\bot_2$,
and ``not necessarily'' in the  cases of $\lambda^\top_2$ and $\lambda^\bot_1$;
see \autoref{ss:nonexistence}.
However, under the assumption of domination (that we will explain next),
all $\sup$'s and $\inf$'s in \eqref{eq:extremal_lambdas} are attained.

\subsection{Domination}

Consider a $2 \times 2$ matrix cocycle $(T,A)$ where $T$ is a homeomorphism.
Suppose that for each $\omega \in \Omega$ we are given a splitting
of $\R^2$ as the sum of two one-dimensional subspaces $e_1(\omega)$, $e_2(\omega)$.
We say that this is a \emph{dominated splitting} with respect to the cocycle $(T,A)$
if the following properties hold:
\begin{itemize}
\item equivariance:
\begin{equation}\label{eq:equivariance}
A(\omega) (e_i(\omega)) = e_i( T \omega) \quad
\text{for all $\omega \in \Omega$ and $i \in \{1,2\}$;}
\end{equation}
\item dominance:  there are constants $c>0$ and $\delta>0$ such that
\begin{equation}\label{eq:domination_via_splitting}
\frac{\| A^{(n)} (\omega) | e_1(\omega) \|}{\| A^{(n)} (\omega) | e_2(\omega) \|} \ge c e^{\delta n} \quad
\text{for all $\omega \in \Omega$ and $n \ge 1$.}
\end{equation}
\end{itemize}

An important property of dominated splittings is that they are always continuous,
that is, $e_1$ and $e_2$, viewed as maps from $\Omega$ to the projective space $\P^1$, are automatically continuous (see e.g.~\cite[\S~B.1]{BDV_book}).

We say that a cocycle is \emph{dominated} if it admits a dominated splitting.
Some authors say that the cocycle is \emph{exponentially separated}, which is perhaps a better terminology. Domination is also sometimes called \emph{projective hyperbolicity}, because it can be expressed in terms of uniform contraction and expansion on the projective space.

As shown in \cite{Yoccoz,BoGo}, a $2\times 2$ cocycle $(T,A)$
is dominated if and only if there are constants $c>0$ and $\delta>0$ such that
\begin{equation}\label{eq:domination_via_non_conformality}
\frac{\|A^{(n)} (\omega)\|}{\fm (A^{(n)} (\omega))}  \ge c e^{\delta n}
\quad \text{for all $\omega \in \Omega$ and $n \ge 0$.}
\end{equation}
Note that the LHS is a measure of ``non-conformality'' of the matrix $A^{(n)} (\omega)$.

If a cocycle is dominated then the Lyapunov exponents \eqref{eq:LE} are always distinct;
actually $\delta>0$ as in \eqref{eq:domination_via_non_conformality}
is a uniform lower bound for the gap between them.
Moreover (see \autoref{ss:preliminaries_general})
the Lyapunov exponents are given by integrals:
\begin{equation} \label{eq:exponents_as_integrals}
\lambda_i(A,\mu) = \int \phi_i \, d\mu \, , \quad \text{where }
\phi_i(\omega) \coloneqq \log \| A(\omega) | e_i(\omega) \| \, , \quad (i=1,2).
\end{equation}

As a consequence of these formulas,
the problem of maximizing or minimizing Lyapunov exponents for dominated cocycles
is equivalent to the optimization of Birkhoff averages of the continuous functions $\phi_i$,
and so many standard results apply (see \cite{Jenkinson}).\footnote{In order to avoid complications, our definition \eqref{eq:extremal_lambdas} of the extremal Lyapunov exponents only considers regular points, as the alert reader has noticed. On the other hand, non-regular points have no effect in the optimization of Birkhoff averages (see \cite{Jenkinson}). Therefore, for dominated cocycles at least, non-regular points have no effect in the optimization of Lyapunov exponents.}
In particular, one can easily show that:
\begin{align*}
\lambda^\top_i(A) &= \lim_{n \to \infty} \frac{1}{n} \sup_{\omega \in \Omega} \log \| A^{(n)} (\omega) | e_i(\omega) \| \, , \\
\lambda^\bot_i(A) &= \lim_{n \to \infty} \frac{1}{n} \inf_{\omega \in \Omega} \log \| A^{(n)} (\omega) | e_i(\omega) \| \, .
\end{align*}
Another consequence is a fact we mentioned already: all $\sup$'s and $\inf$'s that appear in
\eqref{eq:extremal_lambdas} are attained in the dominated case.

\subsection{One-step cocycles}

Fix an integer $k \ge 2$.
Let $\Omega = \{1,\dots,k\}^\Z$ be the space of bi-infinite words on $k$ symbols.
With some abuse of notation, we denote this set by $k^\Z$.
Let $T \colon k^\Z \to k^\Z$ be the shift transformation.

Given a $k$-tuple of matrices $\sA = (A_1,\dots,A_k) \in \GL(2,\R)^k$,
we associate with it the locally constant map
$A \colon k^\Z \to \GL(d,\R)$ given by $A(\omega) = A_{\omega_0}$.
In this case, $(T,A)$ is called a \emph{one-step cocycle},
and the products \eqref{eq:main_products} are simply
$$
A^{(n)} (\omega) = A_{\omega_{n-1}} \cdots A_{\omega_0} \, .
$$
The $k$-tuple of matrices $\sA$ is called the \emph{generator} of the cocycle.
We denote $\lambda^\top_i(\sA) \coloneqq \lambda^\top_i(A)$ and $\lambda^\bot_i(\sA) \coloneqq \lambda^\bot_i(A)$.

We remark that for one-step cocycles, the values $\lambda_1^\top(\sA)$ and $\lambda_1^\bot(\sA)$ can be
alternatively defined in a more elementary way
(without speaking of measures) as:
\begin{align}
\lambda_1^\top(\sA) &= \lim_{n\to \infty} \frac{1}{n} \log \sup_{i_1, \dots, i_n} \| A_{i_n} \dots A_{i_1} \| \, , \label{eq:alternative_top}\\
\lambda_1^\bot(\sA) &= \lim_{n\to \infty} \frac{1}{n} \log \inf_{i_1, \dots, i_n} \| A_{i_n} \dots A_{i_1} \| \label{eq:alternative_bot}
\end{align}
(see \autoref{ss:alternative}).

The numbers $\varrho^\top(\sA) \coloneqq e^{\lambda^\top_1(\sA)}$ and $\varrho^\bot(\sA) \coloneqq e^{\lambda^\bot_1(\sA)}$
are called \emph{joint spectral radius} and \emph{joint spectral subradius}
and constitute an active topic of research: see \cite{Jungers_book}.
Let us remark that that the joint spectral radius is always a continuous function of the matrices, while the joint spectral subradius is not; in fact discontinuities of the latter are related to a lack of domination: see \cite{BMo}.

\subsection{Domination for one-step $\GL(2,\R)$ cocycles}

A one-step cocycle $(T,A)$ is dominated if and only if
the number
\begin{equation}\label{eq:diff_bot}
(\lambda_1-\lambda_2)^\bot(A) \coloneqq  \inf_{\mu \in \cM_T} (\lambda_1(A,\mu) - \lambda_2(A,\mu))
\end{equation}
is positive; see \autoref{ss:alternative} for the (easy) proof.
Let us see still another characterization
of domination for one-step cocycles.

\medskip

The \emph{standard symmetric cone} in $\R^2_* \coloneqq \R^2 \setminus \{0\}$ is
$$
C_+ \coloneqq \{ (x,y) \in \R^2_* ; \; xy \ge 0 \}
$$
A \emph{cone} in $\R^2_*$ is an image of $C_+$ by a linear isomorphism.
A \emph{multicone} in $\R^2_*$  is a disjoint union of finitely many cones.

We say that a multicone $\multicone \subset \R^2_*$ is \emph{strictly forward-invariant}
with respect to $\sA = (A_1, \dots, A_k)$ if
the image multicone $\bigcup_i A_i(\multicone)$
is contained in the interior of~$\multicone$.

For example, if the $A_i$'s have positive entries then the
standard symmetric cone $C_+$
is a strictly forward-invariant multicone for $(A_1, \dots, A_k)$.
For more complicated examples, see \cite{ABY}.

It was proved in \cite{ABY,BoGo} that
\emph{the one-step cocycle generated by $\sA$ is dominated
if and only if $\sA$ has a strictly forward-invariant multicone.}

If $\multicone$ is a multicone, its \emph{complementary multicone} $\multiconedual$ is defined as
the closure (relative to $\R^2_*$) of $\R^2_* \setminus \multicone$.
Notice that if $\multicone$ is strictly forward-invariant with respect to $(A_1, \dots, A_k)$
then $\multiconedual$ is \emph{strictly backwards-invariant}, that is, strictly forward-invariant with respect to
$(A_1^{-1}, \dots, A_k^{-1})$.

\subsection{Mather sets}\label{ss:Mather}

Under the assumptions above, the extremal Lyapunov exponents
``live'' in certain invariant sets:

\begin{maintheorem}\label{t:Mather}
Suppose that the one-step cocycle generated by $\sA \in \GL(2,\R)^k$ is dominated.
For each $\star \in \{\top, \bot\}$,
let $K^\star$ be the union of all supports of measures $\mu \in \cM_T$ such that
$\lambda(\mu) = \lambda_1^\star$.
Then:
\begin{itemize}
	\item $K^\star$ is a compact, nonempty, $T$-invariant set;
	\item any measure $\mu \in \cM_T$ supported in $K^\star$ satisfies $\lambda(\mu) = \lambda^\star$.
\end{itemize}
\end{maintheorem}

In particular, $\lambda_1$-optimizing measures exist.

We call $K^\top$ and $K^\bot$ \emph{upper} and \emph{lower Mather sets}, respectively.
Our upper Mather set corresponds to what Morris~\cite{Morris13} calls a Mather set.
The terminology is coherent with Lagrangian Dynamics,
where Mather sets were first studied in \cite{Mather}.
The existence of Mather sets is related to the Bousch's ``subordination principle'': see \cite{Morris07} for the commutative context, and \cite{CZ} for the subadditive context.

Actually the existence of the upper Mather set is guaranteed
for $1$-step cocycles (in any dimension) without assumptions of domination: see \cite{Morris13}.

The existence of both Mather sets in \autoref{t:Mather}
can be deduced from H\"older continuity of the Oseledets directions
using the usual (commutative) ergodic optimization theory.
However, the proof of \autoref{t:Mather} that we will present
is self-contained and gives extra information which will be useful
in the proof of our major result, \autoref{t:zero} below.

\subsection{Zero entropy}\label{ss:zero}

We say that $\sA = (A_1,\dots,A_k)$ satisfies the \emph{forward NOC (non-overlapping condition)}
if it has a strictly forward-invariant multicone $\multicone \subset \R^2_*$ such that
$$
A_i (\multicone) \cap A_j (\multicone) = \emptyset \quad \text{whenever } i \neq j \, .
$$
We say that  $\sA = (A_1,\dots,A_k)$ satisfies the \emph{backwards NOC}
if $(A_1^{-1}, \dots, A_k^{-1})$ satisfies the forward NOC.
We say that $(A_1,\dots,A_k)$ satisfies the \emph{NOC} if
it satisfies both the forward and the backwards NOC.
An intrinsic characterization of these conditions will be given later (\autoref{p:NOC}).

\begin{remark}\label{r:NOCs}
The forward and the backwards NOC are not equivalent:
for example, if
$$
A_1 \coloneqq \begin{pmatrix} \alpha^{-1} & 0 \\ 0 & \alpha \end{pmatrix}, \quad
A_2 \coloneqq \begin{pmatrix} \beta^{-1}  & 0 \\ 1 & \beta   \end{pmatrix}, \quad
\text{with $\alpha>0$, $\beta>0$ sufficiently small}
$$
then $(A_1,A_2)$ satisfies the forward NOC but not the backwards NOC. 

The former statement is easy to check directly. For the latter, note that the actions of $A_1^{-1}$ and $A_2^{-1}$ on the projective space $\P^1$ have a common attracting fixed point, namely the line spanned by $(0,1)$. Hence, however we choose the invariant multicone $M$ for $\{A_1^{-1}, A_2^{-1}\}$, it must contain this line. This implies that $A_1^{-1} M \cap A_2^{-1} M \neq \emptyset$. 
\end{remark}

The main result of this paper is the following:

\begin{maintheorem}\label{t:zero}
For every $k$ and every $\sA \in \GL(2,\R)^k$,
if the one-step cocycle generated by $\sA$ is dominated
and satisfies the NOC then
the restriction of the shift map to either Mather set $K^\top$ or $K^\bot$
has zero topological entropy.
\end{maintheorem}

The conclusion of the theorem means that for each $\star \in \{ \top, \bot\}$,
the number $w^\star (\ell)$  of words of length $\ell$ in the alphabet $\{1, \dots, k\}$
that can be extended to a bi-infinite word in
the Mather set $K^\star$
is a subexponential function of $\ell$, that is,
\begin{equation}\label{eq:subexp}
\lim_{\ell \to \infty} \frac{1}{\ell} \log w^\star(\ell) = 0.
\end{equation}
(see \cite[p.~265--266]{Petersen}).

There are examples where \autoref{t:zero}
applies and $K^\top$ is non-discrete:
In the family of examples given in \cite{BM} where a maximizing measure is Sturmian non-periodic,
the NOC condition holds for some choices of the parameters.\footnote{See \cite{JP} and references therein for more examples of non-periodic Lyapunov-maximizing measures.}

There are also examples where \autoref{t:zero}
applies and either $K^\top$ or $K^\bot$ is not uniquely ergodic: see \autoref{ss:non_uniqueness}.

\subsection{Positive entropy}

As a counterpoint to \autoref{t:zero},
we will exhibit non-trivial situations where
$\lambda_1$-minimizing measures with \emph{positive} entropy exist.

A cocycle $(T,A)$ is called \emph{uniformly hyperbolic} if
it has an equivariant splitting into two subbundles, one being uniformly expanding
and the other being uniformly contracting.
Any uniformly hyperbolic cocycle is dominated,
and the converse holds for $\SL(2,\R)$-cocycles.

\begin{maintheorem}\label{t:positive}
Fix $k \ge 2$ and let $T$ be the full shift in $k$ symbols.
There exists an open and dense subset $\cU$ of $\SL(2,\R)^k$
such that for every $\sA \in \cU$,
\begin{enumerate}
\item either the one-step cocycle over $T$ generated by $\sA$
is uniformly hyperbolic;
\item
or there exists a compact $T$-invariant set $K \subset k^\Z$
of positive topological entropy and such that the norms $\|A^{(n)}(\omega)\|$
are uniformly bounded over $(\omega,n) \in K \times \Z$.
\end{enumerate}
\end{maintheorem}

Notice that in the first case we have $\lambda_1^\bot(\sA) > 0$,
while in the second case
it follows from the entropy variational principle (see \cite[p.~269]{Petersen})
that there exists a measure $\mu \in \cM_T$ such that
$h_\mu(T) > 0$ and
$\lambda_1(\sA, \mu) = 0$.

For a nonlinear version of \autoref{t:positive}, see \cite[Theorem~2]{BoBoDi}.

\subsection{Organization of the paper and overview of the proofs}

In \autoref{s:preliminaries} we collect basic fact about dominated cocycles.

A standard procedure to solve ergodic optimization problems
is to look for a change of variables under which the
optimizing orbits become evident, or ``revealed''.
Following this idea, in \autoref{s:Barabanov}
we construct what we call ``Barabanov functions''
(in analogy to the Barabanov norms from joint spectral radius theory),
and immediately use them to prove the existence
of the Mather sets (\autoref{t:Mather}).

In \autoref{s:geometry} we use the Barabanov functions to prove that
the directions of the dominated splitting for points on the Mather sets
must obey severe geometrical obstructions, which in turn imply that
one direction uniquely determines the other, with an at most countable number of exceptions.
Using this property, we prove \autoref{t:zero} in \autoref{s:zero}.

Let us mention that some of the key parts in sections~\ref{s:Barabanov} and \ref{s:geometry} were inspired by ideas from the paper of Bousch and Mairesee \cite{BM}.
Nevertheless, we do not use their results directly.

The simpler proof of \autoref{t:positive} is given in \autoref{s:positive},
and is independent of the previous sections.

In \autoref{s:appendix} we present complementary information,
including counterexamples showing the limits of our results
and alternative definitions for some of the concepts we have discussed.
In the final subsection \ref{ss:open},
we pose a few problems and suggest some directions for future research.

\section{Preliminaries: Basic facts about $2 \times 2$ dominated cocycles}\label{s:preliminaries}

In this section we collect some simple facts 
about dominated cocycles that will be needed in the sequel.
Even though these facts are standard,
for the convenience of the reader we will provide some of the proofs.

\subsection{General cocycles}\label{ss:preliminaries_general}

In this subsection, let $\Omega$ be a compact metric space,
let $T \colon \Omega \to \Omega$ be a homeomorphism,
let $A \colon \Omega \to \GL(2,\R)$ be a continuous map,
and assume that the cocycle $(T,A)$ has a dominated splitting
into directions $e_1$, $e_2$.

\begin{lemma}\label{l:dom_easy}
If $\omega\in\Omega$ and $x \in \R^2 \setminus e_2(\omega)$ then
$$
0<\lim_{n \to \infty} \frac{\|A^{(n)}(\omega) x\|}{\|A^{(n)}(\omega) | e_1(\omega) \|}<\infty, \quad
\lim_{n \to \infty} \angle\big(A^{(n)}(\omega) x, e_1(T^n \omega) \big) = 0 \, .
$$
\end{lemma}

\begin{proof}
Easy and left to the reader.
\end{proof}

\begin{proposition}\label{p:dom_uniqueness}
The dominated splitting is unique.
\end{proposition}
	
\begin{proof}
Let $f_1 \oplus f_2$ be another dominated splitting for the cocycle $(T,A)$.
Fix $\omega \in \Omega$ and
consider the product
$$
\frac{\|A^{(n)}(\omega) | e_1(\omega)\|}{\|A^{(n)}(\omega) | e_2(\omega) \|} \cdot
\frac{\|A^{(n)}(\omega) | f_2(\omega)\|}{\|A^{(n)}(\omega) | e_1(\omega) \|} \cdot
\frac{\|A^{(n)}(\omega) | f_1(\omega)\|}{\|A^{(n)}(\omega) | f_2(\omega) \|} \cdot
\frac{\|A^{(n)}(\omega) | e_2(\omega)\|}{\|A^{(n)}(\omega) | f_1(\omega) \|}
= 1 \, .
$$
Then the first and the third factors
tend to $\infty$ as $n \to \infty$.
If $f_2(\omega) \neq e_2(\omega)$ then by \autoref{l:dom_easy}
the second and the fourth factors have nonzero finite limits as  $n \to \infty$,
which is a contraction.
This shows that $f_2 = e_2$.
The same reasoning for the inverse cocycle shows that $f_1 = e_1$.
\end{proof}
	
\begin{lemma}\label{l:dom_comparison}
There exists $C>1$ such that
\begin{equation}\label{eq:dom_comparison}
C^{-1} \| A^{(n)} (\omega) | e_2(\omega) \| \le
\fm (A^{(n)} (\omega)) \le
\|A^{(n)} (\omega)\| \le
C \| A^{(n)} (\omega) | e_1(\omega) \|
\end{equation}
for any $\omega \in \Omega$ and $n \ge 0$.
\end{lemma}

\begin{proof}
This is a consequence of the fact that the angles between the
directions of the dominated splitting are uniformly bounded from below.
The details are left to the reader.
\end{proof}

It follows that the Lyapunov exponents defined by \eqref{eq:LE}
can be determined from the restriction of the cocycle to the directions that form the
dominated splitting:

\begin{corollary}\label{c:reduce_1D}
If the cocycle  $(T,A)$ is dominated then, for any $i \in \{1,2\}$,
\begin{equation}\label{eq:reduce_1D}
\lambda_i(A,\omega)
= \lim_{n \to \infty} \frac{1}{n} \log\|A^{(n)}(\omega) | e_i(\omega) \|
\end{equation}
for every $\omega \in \Omega$ such that at least one of these quantities is well-defined.
\end{corollary}

\begin{proof}
Use \autoref{l:dom_comparison} together with the obvious estimates:
\[
\| A^{(n)} (\omega) | e_2(\omega) \| \ge \fm (A^{(n)} (\omega)) \quad \text{and} \quad
\| A^{(n)} (\omega) | e_1(\omega) \| \le \|A^{(n)} (\omega)\| \, . \qedhere
\]
\end{proof}

Notice that the RHS of \eqref{eq:reduce_1D} is a limit of Birkhoff averages,
so the integral formulas \eqref{eq:exponents_as_integrals} follow.

\begin{remark}
Actually \autoref{l:dom_comparison} implies that
the dominated splitting coincides with the Oseledets splitting wherever
the latter is defined.
The properties expressed by \autoref{c:reduce_1D} and formulas~\eqref{eq:exponents_as_integrals} hold in general for Oseledets splittings.
\end{remark}

\subsection{One-step cocycles}\label{ss:preliminaries_1_step}

Let us fix some notation.
The projective space of $\R^2$ is denoted by $\P^1$. 
Given $x \in \R^2_*$, let $\project{x}$ denote the unique line in $\P^1$ containing $x$.
For $x$, $y\in \R^2_*$, 
we denote by $\angle(x,y)$ the angle between the lines $x'$, $y'$;
after restriction to $\P^1$ this gives us the usual metric there.
Given a linear isomorphism $L$ of $\R^2$, let $\project{L}$
the self-map of $\P^1$ defined by $\project{L}(\project{u}) = \project{(L(u))}$.
If $\multicone \subset \R^2_*$ is a multicone then
let $\project{\multicone} \coloneqq \big\{ \project{x} \in \P^1 ; \; x \in \multicone \big\}$.

\smallskip

The following result provides an useful ``adapted metric'' on the multicone.
Similar constructions appear in \cite{ABY,BoGo,BMo}.

\begin{lemma}\label{l:adapted}
Assume that $(A_1, \dots, A_k) \in \GL(2,\R)^k$ generates a dominated one-step cocycle,
and let $\multicone \subset \R^2$ be a strictly forward-invariant multicone.
There exist a metric $d$ on the projectivization $\project{M}$
and constants $c_1>1$ and $0 < \tau < 1$ such that
for all $x$, $y \in \multicone$, we have
\begin{align}
&d \left( \project{A_i}\project{x} , \project{A_i}\project{y}\right)
\le \tau d \left( \project{x} , \project{y}\right)
\quad \text{for all } i \in \{1,\dots,k\} ,
\label{eq:metric_1}
\\
&c_1^{-1} \angle \left( x, y \right) \le
d \left( \project{x} , \project{y}\right) \le
c_1 \angle \left( x, y \right).
\label{eq:metric_2}
\end{align}
\end{lemma}

\begin{proof}
By a compactness argument, there exists an open neighborhood $U$ of $\project{\multicone}$
in $\P^1$ such that $\project{A_i}(U) \subset \project{\multicone}$ for all $i \in \{1,\dots,k\}$.
We can assume that each connected component of $U$ contains exactly one connected component of~$\project{M}$.

Endow each connected component of $U$ with its Hilbert metric,
and restrict it to the corresponding connected component of $\project{M}$.
We use the same letter $d$ to denote all those metrics.
Rescaling if necessary, we can assume that $d \le 1/2$ whenever defined.
Moreover, there are constants $c_1>1$ and $0 < \tau < 1$ such that
properties~\eqref{eq:metric_1} and \eqref{eq:metric_2}
hold whenever $\project{x}$ and $\project{y}$
are in the same connected component of $\project{\multicone}$.

Given $\project{x}$, $\project{y} \in \project{\multicone}$,
define $\ell(\project{x}, \project{y})$ as
the least integer $n \ge 0$ with the property
that for all $\omega \in k^\Z$, the directions
$\project{A^{(n)}(\omega)}\project{x}$ and $\project{A^{(n)}(\omega)}\project{y}$
belong to the same connected component of $\project{M}$. Notice that such a number always exists
and is not bigger than the square of the number of components of $M'$; this follows by a pigeon-hole argument using the fact that no linearly-induced map can have two disjoint strictly invariant intervals.


The function $\ell$
has the following property:
$$
\ell \left( \project{A_i}\project{x} , \project{A_i}\project{y}\right)
\le
\max \left( \ell \left( \project{x} , \project{y}\right) - 1, 0 \right),
\quad \text{for all } i \in \{1,\dots,k\} ,
$$
and satisfies an ultrametric inequality:
$$
\ell \left( \project{x} , \project{y}\right)
\le
\max \left( \ell \left( \project{x} , \project{z} \right) , \ell \left( \project{y} , \project{z} \right)  \right).
$$
We now extend $d$ by setting
$d \left( \project{x} , \project{y}\right) \coloneqq \ell \left( \project{x} , \project{y}\right)$
if $\project{x}$ and $\project{y}$ are in different connected components of $\project{\multicone}$.
Then $d$ is a distance function.
Moreover, increasing $c_1$ and $\tau$ if necessary, properties
\eqref{eq:metric_1} and \eqref{eq:metric_2} are satisfied.
\end{proof}

\begin{proposition}\label{p:nested}
Assume that $(A_1, \dots, A_k) \in \GL(2,\R)^k$
generates a dominated one-step cocycle.
Let $e_1$, $e_2 \colon k^\Z \to \P^1$ be the invariant directions forming the dominated splitting,
and let $M \subset \R^2_*$ be a strictly forward-invariant multicone,
and let $\multiconedual$ be the (strictly backwards-invariant) complementary multicone.
Then for any $\omega \in k^\Z$ we have
\begin{align}
\{ e_1(\omega) \} &=
\bigcap_{n=1}^\infty \project{A_{\omega_{-1}}} \cdots \project{A_{\omega_{-n}}} (\project{\multicone}) \, ,  \label{eq:nested_1}
\\
\{ e_2(\omega) \} &=
\bigcap_{n=1}^\infty (\project{A_{\omega_{n-1}}} \cdots \project{A_{\omega_{0}}})^{-1} (\project{\multiconedual}) \, . \label{eq:nested_2}
\end{align}
\end{proposition}

In particular, $e_1(\omega) \in \project{\multicone}$ and $e_2(\omega) \in \project{\multiconedual}$. 
To prove this proposition,
we will use \autoref{l:adapted} and the following fact:

\begin{lemma}\label{l:angle_derivative}
If $B \in \GL(2,\R)$ and $x\in \R^2_*$ then
$$
\lim_{y \to x} \frac{\angle(Bx,By)}{\angle(x,y)} =
|\det B| \left(\frac{\|Bx\|}{\|x\|}\right)^{-2} \, .
$$
\end{lemma}

\begin{proof}
Up to an error $o(\|x-y\|)$,
the triangle with vertices $0$, $x$, $y$ has area $\|x\|^2 \angle(x,y)$,
while the triangle with vertices $0$, $Bx$, $By$ has area $\|Bx\|^2 \angle(Bx,By)$. Since the linear map $B$ expands areas by a factor $|\det B|$, the lemma follows.
\end{proof}

\begin{proof}[Proof of \autoref{p:nested}]
Assume given a multicone $M$ for $(A_1,\dots,A_k)$.
We will show that formulas \eqref{eq:nested_1}, \eqref{eq:nested_2} define
directions forming a dominated splitting;
therefore the proposition will follow from the uniqueness of dominated splittings
(\autoref{p:dom_uniqueness}).

Indeed, for every $\omega \in k^\Z$, the RHS of \eqref{eq:nested_1}
is a nested intersection of compact subsets of $\P^1$
whose diameters, by \autoref{l:adapted}, converge to $0$.
Therefore the intersections contains a single point in $\P^1$; call it $e_1(\omega)$.
Analogously we define $e_2(\omega)$ using \eqref{eq:nested_2}.
This pair of directions forms a continuous splitting with the equivariance property~\eqref{eq:equivariance}.
Let us check that this splitting is dominated.
By \autoref{l:adapted}, there exist constants $c_1>1$, $0<\tau<1$ such that
$$
\frac{\angle \big( A^{(n)}(\omega)x,A^{(n)}(\omega)y \big)}{\angle(x,y)} \le c_1^2 \tau^n
\quad \text{for all } \omega\in k^\Z, \ x,y\in M,\ n\ge 0 \, .
$$
Taking $x \in e_1(\omega)$ and using \autoref{l:angle_derivative},
we conclude that
$$
\| A^{(n)}(\omega) | e_1(\omega) \| \ge c_1^{-1} |\det A^{(n)}(\omega)|^{1/2} \tau^{-n/2}
\quad \text{for all } \omega\in k^\Z, \ n\ge 0 \, .
$$
By an analogous argument, there exist constants $\bar{c}_1>1$, $0<\bar{\tau}<1$ such that
$$
\| A^{(n)}(\omega) | e_2(\omega) \| \le \bar{c}_1 |\det A^{(n)}(\omega)|^{1/2} \bar{\tau}^{n/2}
\quad \text{for all } \omega\in k^\Z, \ n\ge 0 \, .
$$
Therefore the domination property~\eqref{eq:domination_via_splitting}
holds for appropriate constants $c$, $\delta$.
As explained before, the proposition follows from uniqueness of dominated splittings.\footnote{Incidentally, we have just proved that the existence of a strictly forward-invariant multicone implies domination. For the proof of the converse, see \cite{ABY,BoGo}.}
\end{proof}

Let $\Z_-$ (resp.\ $\Z_+$) be the set of negative (resp.\ nonnegative) integers.
Define the following projections:
\begin{alignat}{2}
\pi_- \colon k^\Z &\to k^{\Z_-} \, , &\qquad \pi_-(\omega) &= (\dots, \omega_{-2},\omega_{-1}) \, ,
\label{eq:pi-}
\\
\pi_+ \colon k^\Z &\to k^{\Z_+} \, , &\qquad \pi_+(\omega) &= (\omega_0, \omega_1,\dots) \, .
\label{eq:pi+}
\end{alignat}

As a straightforward consequence of \autoref{p:nested}, we have:

\begin{corollary}\label{c:factor_projections}
Assume that $(A_1, \dots, A_k) \in \GL(2,\R)^k$ generates a dominated one-step cocycle.
Then, for each $\omega \in k^\Z$, the invariant directions $e_1(\omega)$ and $e_2(\omega)$
forming the dominated splitting depend only on
$\pi_-(\omega)$ and $\pi_+(\omega)$, respectively,
and therefore there exist continuous maps $\tilde{e}_1$, $\tilde{e}_2$
such that the following diagrams commute:
\begin{equation}\label{eq:diagrams}
\begin{tikzcd}
k^\Z \arrow{d}[swap]{\pi_-}\arrow{r}{e_1} & \P^1 \\
k^{\Z_-} \arrow{ru}[swap]{\tilde{e}_1}
\end{tikzcd}
\qquad\qquad
\begin{tikzcd}
k^\Z \arrow{d}[swap]{\pi_+}\arrow{r}{e_2} & \P^1 \\
k^{\Z_+} \arrow{ru}[swap]{\tilde{e}_2}
\end{tikzcd}
\end{equation}	
\end{corollary}

The NOC conditions introduced in \autoref{ss:zero} can be characterized in terms of the invariant directions as follows:

\begin{proposition}\label{p:NOC}
Assume that $(A_1, \dots, A_k) \in \GL(2,\R)^k$ generates a dominated one-step cocycle.
Let  $\tilde{e}_1$, $\tilde{e}_2$ be the maps given by \autoref{c:factor_projections}.
Then:
\begin{itemize}
\item $\tilde{e}_1$ is one-to-one if and only if the forward NOC holds, in which case $e_1 ( k^\Z )$ is a Cantor set.
\item $\tilde{e}_2$ is one-to-one if and only if the backwards NOC holds, in which case $e_2 ( k^\Z )$ is a Cantor set.
\end{itemize}
\end{proposition}

\begin{proof}
We only consider the first statement, since the second one is analogous.
The ``if'' part is a direct consequence of formula \eqref{eq:nested_1}.
Conversely, assume that $\tilde{e}_1$ is one-to-one.
Consider a strictly forward-invariant multicone $M$ and the adapted metric $d$ on $M'$ given by \autoref{l:adapted}.
It follows from equivariance \eqref{eq:equivariance} that the compact sets $A_i'(e_1(k^\Z))$ are disjoint; in particular their $d$-distances from one another are bigger than some $\epsilon>0$. Consider the closed $\epsilon/2$-neighborhood of $e_1(k^\Z)$. This set is a projectivized multicone (because each component has size at least $\epsilon$, hence there are finitely many of them), is strictly forward-invariant (by \eqref{eq:metric_1}), and its images under the maps $A_i'$ do not overlap. So the forward NOC holds.
\end{proof}

\section{Barabanov functions and Mather sets}\label{s:Barabanov}

\subsection{Statements}

A \emph{Barabanov norm} for a compact set $\sA$ of $d \times d$ matrices is a norm $\tribar{\mathord{\cdot}}$ on $\R^d$ such that
\begin{equation}\label{eq:true_Barabanov}
\max_{A \in \sA} \, \tribar{A x} = \varrho^\top(\sA) \, \tribar{x} \quad \text{for all } x \in \R^d,
\end{equation}
where $\varrho^\top(\sA) = e^{\lambda_1^\top(\sA)}$ is the joint spectral radius of $\sA$.\footnote{The definition of joint spectral radius of a (non-necessarily finite) set of matrices is analogous to the definition given previously for $k$-tuples of matrices.}
It is known that a Barabanov norm exists whenever $\sA$ is irreducible (i.e., has no nontrivial invariant subspace): see \cite{Barabanov,Wirth}.

For definiteness, let us consider finite sets $\sA \subset \GL(2,\R)$.
One may wonder about the existence of a version of the Barabanov for the
joint spectral subradius $\varrho^\bot(\sA) = e^{\lambda_1^\bot(\sA)}$,
that is, a norm $\tribar{\mathord{\cdot}}$ such that
\begin{equation}\label{eq:min_Barabanov}
\min_{A \in \sA} \, \tribar{A x} = \varrho^\bot(\sA) \, \tribar{x} \quad \text{for all } x \in \R^2.
\end{equation}
Unfortunately, no such norm can in general exist, even assuming irreducibility of $\sA$.
For example, if the cocycle is such that $\lambda_2^\top(\sA) < \lambda_1^\bot(\sA)$
then applying relation \eqref{eq:min_Barabanov}
to the orbit of a nonzero vector in the second Oseledets direction $e_2$
we reach a contradiction.

This example shows that if such a ``minimizer Barabanov norm'' exists,
relation \eqref{eq:min_Barabanov} cannot hold for all vectors,
but only for vectors away from the $e_2$-directions.
In general, the set of $e_2$-directions can be large or even the whole $\P^1$,
but for dominated cocycles it is a proper compact subset of $\P^1$.

As we show in this section, under the assumption of
domination it is indeed possible to construct an object that retains
the most useful properties of (the logarithm of) a ``minimizer Barabanov norm''.
For convenience, we simultaneously consider both
the maximizer and minimizer cases:

\begin{theorem}\label{t:Barabanov}
Assume that $(A_1, \dots, A_k) \in \GL(2,\R)^k$ generates a dominated one-step cocycle,
and let $\multicone \subset \R^2$ be a strictly forward-invariant multicone.
Then there exist functions
$$
p^\top \colon \multicone \to \R
\quad \text{and} \quad
p^\bot \colon \multicone \to \R
$$
with the following properties:
\begin{itemize}
\item extremality: for all $x \in \multicone$,
\begin{align}
\max_{i \in \{1,\dots, k\}} p^\top (A_i x) &= p^\top (x) + \lambda_1^\top  \, , \label{eq:ptop} \\
\min_{i \in \{1,\dots, k\}} p^\bot (A_i x) &= p^\bot (x) + \lambda_1^\bot  \, ; \label{eq:pbot}
\end{align}
\item $\log$-homogeneity:
for all $\star \in \{\top, \bot\}$, $x \in \multicone$, and $t \in \R_*$,
\begin{equation}
p^\star(tx) = p^\star(x)+ \log|t|   \, ; \label{eq:p_homog}
\end{equation}
\item regularity: there exists $c_0>0$ such that for all  $\star \in \{\top, \bot\}$
and $x$, $y \in \multicone$,
\begin{equation}
|p^\star(x) - p^\star(y)| \le c_0 \angle(x,y) + |\log \|x\| - \log\|y\| | \, . \label{eq:p_Lip}
\end{equation}
\end{itemize}
\end{theorem}

Related functions were used by Bousch and Mairesse \cite[\S~2.1]{BM}.
Our construction combines their techniques with properties of multicones
and the Hilbert metric.
A higher-dimensional version of our construction (based on a preliminary version of this paper)
was obtained in \cite{BMo}.

Let us also mention that similar constructions also play an important
role on ergodic optimization,
action minimization in Lagrangian dynamics,
and optimal control: see \cite{BM} and references therein.

Back to the context of products of matrices, Morris \cite{Morris13} has studied relations
between the Barabanov norm \eqref{eq:true_Barabanov} and the upper Mather set.

In this paper, the upper and lower Barabanov functions are the main tools in the analysis of the upper and lower Mather sets.

\subsection{Proofs}

In the following proof of \autoref{t:Barabanov},
we will also establish some facts that are necessary for
the subsequent proof of \autoref{t:Mather}.

\begin{proof}[Proof of \autoref{t:Barabanov}]
For each $i \in \{1,\dots,k\}$, define $h_i \colon \P^1 \to \R$ by
$$
h_i (\project{x}) \coloneqq \log \frac{\|A_i x\|}{\|x\|} \, ,
$$
where, as before, a prime denotes projectivization and $\| \mathord{\cdot} \|$ denotes the Euclidian metric.
Fix a constant $c_2 > 0$ such that
$$
\left| h_i(\project{x}) - h_i(\project{y}) \right| \le c_2 \angle(x,y) \quad
\text{for all $x$, $y \in \R^2_*$ and all $i$.}
$$

Let $\multicone$ be a strictly forward-invariant multicone for $(A_1, \dots, A_k)$,
and let $d$ be the metric on the projectivization $\project{\multicone}$
given by \autoref{l:adapted}.
Let $\B$ be the vector space of continuous functions from $\project{\multicone}$ to $\R$,
endowed with the uniform (supremum) distance $| \mathord{\cdot} |_\infty$.
Let $c_3 \coloneqq c_1 c_2 / {(1-\tau)}$ and let
$\K \subset \B$ be the set of functions that are $c_3$-Lipschitz with respect to $d$.

For each function $f \in \K$, define two functions $T^\star f \colon \project{\multicone} \to \R$
(where $\star \in \{\top, \bot\}$) by
\begin{align*}
(T^\top f) (\project{x}) &\coloneqq \max_{i \in\{1,\dots,k\}} \left[ f\left( \project{A_i}\project{x}\right) + h_i(\project{x}) \right] \, , \\
(T^\bot f) (\project{x}) &\coloneqq \min_{i \in\{1,\dots,k\}} \left[ f\left( \project{A_i}\project{x}\right) + h_i(\project{x}) \right] \, .
\end{align*}
We claim that $T^\star f \in \K$.
Indeed, for all $\project{x}$, $\project{y} \in \project{\multicone}$, we have
\begin{align*}
\left| (T^\star f) (\project{x}) - (T^\star f) (\project{y}) \right|
&\le \max_i
\Big| \left[ f\left( \project{A_i}\project{x}\right) + h_i(\project{x}) \right]
    - \left[ f\left( \project{A_i}\project{y}\right) + h_i(\project{y}) \right] \Big| \\
&\le \max_i  \left| f\left( \project{A_i}\project{x}\right) - f\left( \project{A_i}\project{x}\right) \right|
   + \max_i \left| h_i(\project{x}) - h_i(\project{y}) \right| \\
&\le c_3 \max_i d \left( \project{A_i}\project{x} , \project{A_i}\project{y }\right) +  c_2 \angle(x,y) \\
&\le c_3 \tau d(\project{x},\project{y}) + c_1 c_2 d(\project{x},\project{y}) \\
&=   c_3 d(\project{x},\project{y}) \, .
\end{align*}
Thus we have defined maps $T^\star \colon \K \to \K$.
Next, we claim that these maps are continuous.
Indeed, for all $f$, $g \in \K$, 
we have
\begin{align*}
|T^\star f - T^\star g|_\infty
&=   \sup_{\project{x} \in \project{\multicone}} \left| (T^\star f) (\project{x}) - (T^\star g) (\project{x}) \right| \\
&\le \sup_{\project{x} \in \project{\multicone}} \max_i  \left| f\left( \project{A_i}\project{x}\right) - g \left( \project{A_i}\project{x}\right) \right| \\
&\le |f-g|_\infty \, .
\end{align*}

Let $\hat \B$ be the quotient of the space $\B$ by the subspace of constant functions;
it is a Banach space endowed with the quotient norm
$|\hat f|_\infty \coloneqq \inf \{|f|_\infty ; \; \pi(f)=\hat{f}\}$,
where $\pi \colon \B \to \hat \B$ denotes the quotient projection.
By the Arzel\`a--Ascoli theorem, the convex set $\hat\K \coloneqq \pi(\K)$ is compact.
Since $T^\star$ commutes with the addition of a constant, there exists a map
$\hat T^\star \colon \hat \K \to \hat \K$ such that $\pi \circ T^\star = \hat T^\star \circ \pi$.
The map $\hat T^\star$ is continuous;
in particular, by the Schauder theorem, it has a fixed point $\hat{f}_0^\star$.
This means that there exist $f_0^\star \in \K$ and $\beta^\star \in \R$ such that
$T^\star f_0^\star = f_0^\star + \beta^\star$.
Define
$$
p^\star(x) \coloneqq f_0^\star(\project{x}) + \log \|x\| \quad \text{for all } x \in \multicone.
$$
Note that for every $x \in \multicone$,
the following properties hold: property \eqref{eq:p_homog},
\begin{align}
\max_{i \in \{1,\dots, k\}} p^\top (A_i x)  &= p^\top (x) + \beta^\top  \, , \label{eq:disguise_top} \\
\min_{i \in \{1,\dots, k\}} p^\bot (A_i x)  &= p^\bot (x) + \beta^\bot  \, , \label{eq:disguise_bot}
\end{align}
and
\begin{equation}\label{eq:easy_comparison}
\Big|  p^\star (x) - \log \|x\| \Big| \le c_4
\quad
\end{equation}
where $c_4 \coloneqq \max \big( |f_0^\top|_\infty \, , |f_0^\bot|_\infty \big)$.

Taking $c_0 = c_1 c_3$, we see that property \eqref{eq:p_Lip} holds.
To complete the proof of \autoref{t:Barabanov} we need to show that
the numbers $\beta^\top$ and $\beta^\bot$
that appear in \eqref{eq:disguise_top} and \eqref{eq:disguise_bot}
are respectively equal to the numbers
$\lambda_1^\top$ and $\lambda_1^\bot$ that appear in \eqref{eq:ptop} and \eqref{eq:pbot}.
As we prove these equalities, we will also establish
some facts that will be useful in the forthcoming proof of \autoref{t:Mather}.

\medskip

For each $\star \in \{ \top, \bot \}$, let us define
a function $\psi^\star \colon k^\Z \to \R$ by
\begin{equation}\label{eq:def_psi}
\psi^\star (\omega) \coloneqq p^\star (A_{\omega_0} x) - p^\star(x) \, ,
\quad \text{where $x \in e_1(\omega) \setminus \{0\}$ is arbitrary.}
\end{equation}
By \autoref{p:nested} we have $e_1(\omega) \subset M$, so
the expression above makes sense,
and by \eqref{eq:p_homog} it does not depend on the choice of $x$;
thus $\psi^\star$ is well-defined.
Moreover it is a continuous function.

By equivariance of the $e_1$ direction,
for every $\omega \in k^\Z$, $x \in e_1(\omega) \setminus \{0\}$, and $n\ge 1$ we have
$$
p^\star (A^{(n)}(\omega) x) - p^\star(x) = \sum_{j=0}^{n-1} \psi^\star(T^j \omega) .
$$
Letting $\phi_1(\omega) \coloneqq \log \| A(\omega) | e_1(\omega) \|$,
it follows from \eqref{eq:easy_comparison} that
$$
-2c_4 \le \sum_{j=0}^{n-1} \psi^\star(T^j \omega)  - \sum_{j=0}^{n-1} \phi_1^\star(T^j \omega) \le 2c_4 \, .
$$
Integrating with respect to some $\mu \in \cM_T$,
dividing by $n$, and making $n \to \infty$,
we conclude that $\int \psi^\star \, d\mu = \int \phi_1 \, d\mu$.
Recalling the integral formula \eqref{eq:exponents_as_integrals}
(proved in \autoref{ss:preliminaries_general}),
we conclude that
$$
\lambda_1(\mu) = \int \psi^\star \, d\mu \quad \text{for any $\mu \in \cM_T$.}
$$
On the other hand, by \eqref{eq:disguise_top} and \eqref{eq:disguise_bot}, we have
$$
\psi^\top \le \beta^\top \quad \text{and} \quad
\psi^\bot \ge \beta^\bot \, ,
$$
which in particular implies that
\begin{equation}\label{eq:sandwich}
\beta^\bot \le \lambda_1^\bot \le \lambda_1^\top \le \beta^\top \, .
\end{equation}
Moreover, for any $\mu \in \cM_T$,
we have $\lambda_1(\mu) = \beta^\star$ if and only if
$\psi^\star = \beta^\star$ $\mu$-almost everywhere,
or equivalently, if the $T$-invariant set
\begin{equation}\label{eq:def_L}
L^\star \coloneqq \{ \omega \in k^\Z ; \; \psi^\star(T^n \omega) = \beta^\star \ \forall \, n \in \Z \}
\end{equation}
has total $\mu$-measure.

We will show that $L^\star$ is compact and nonempty.
We begin by showing the following:

\begin{claim}\label{cl:choose_future}
For any $\omega_- \in k^{\Z_-}$
there exists $\omega_+ \in k^{\Z_+}$ such that
if $\omega = \omega_- \omega_+$ is concatenation of $\omega_-$ and $\omega_+$
then $\psi^\star(T^n \omega) = \beta^\star$ for all $n \ge 0$.
\end{claim}

\begin{proof}[Proof of the claim]
Recall from \autoref{c:factor_projections}
that a semi-infinite word
$\omega_- = (\dots, \omega_{-2}, \omega_{-1})$,
determines a direction $\tilde{e}_1(\omega_-)$,
and by \eqref{eq:ptop} or \eqref{eq:pbot}
there exists a letter $\omega_0$ such that
$\psi^\star (\omega)$ (which is well-defined even if $\omega_1$, $\omega_2$, \dots are still undefined)
equals $\beta^\star$.
Next we consider the shifted word $(\dots, \omega_{-1}, \omega_{0})$,
and repeat the reasoning above to find $\omega_1$ such that $\psi^\star(T \omega) = \beta^\star$.
Continuing by induction, we find the desired $\omega_+$, thus proving the claim.
\end{proof}

Let $L^\star_+$ be the set of $\omega \in k^\Z$ such that
$\psi^\star(T^n \omega) = \beta^\star$ for all $n \ge 0$,
which by \autoref{cl:choose_future} is nonempty.
Since $L^\star_+$ is compact and contains $T(L^\star_+)$, the set
$L^\star = \bigcap_{n \ge 0} T^n (L^\star_+)$ is compact and nonempty,
as announced.
In particular, there exists at least one $T$-invariant probability measure $\mu^\star$
supported on $L^\star$,
and so with $\lambda_1(\mu^\star) = \beta^\star$.
Together with \eqref{eq:sandwich} this implies that $\beta^\star = \lambda_1^\star$.
So  \eqref{eq:ptop} and \eqref{eq:pbot}
respectively follow from \eqref{eq:disguise_top} and \eqref{eq:disguise_bot}
and the proof of \autoref{t:Barabanov} is complete.
\end{proof}

\begin{proof}[Proof of \autoref{t:Mather}]
For each $\star \in \{ \top, \bot\}$,
let $\cM_T^\star$ be the set of measures $\mu \in \cM_T$ such that $\lambda(\mu) = \lambda_1^\star$.
We have seen in the proof of \autoref{t:Barabanov} that there exists a nonempty compact $T$-invariant set
$L^\star$ such that $\mu \in \cM_T^\star$ if and only if $\supp \mu \subset L^\star$.

Define the Mather set $K^\star$ as the union of the supports of
all measures $\mu$ in $\cM_T^\star$, so
\begin{equation}\label{eq:Mather_included}
K^\star \subset L^\star \, .
\end{equation}
To show that $K^\star$ is a compact set, we follow an argument from \cite{Morris13}.
The set of all Borel probabilities on $k^\Z$ with the usual weak-star topology
is metrizable and compact, and $\cM_T$ is a compact subset.
Since $L^\star$ is compact, using Urysohn's lemma we see that
the set $\cM_T^\star$ is also compact.
In particular, it has a countable dense sequence $(\nu_n^\star)$.
Consider $\nu^\star \coloneqq \sum 2^{-n} \nu_n^\star$, which is an element of $\cM_T^\star$.
It is then easy to show that $\supp \nu^\star = K^\star$,
which in particular shows that $K^\star$ is compact.

The remaining assertions in \autoref{t:Mather} are now obvious, and the proof is complete.
\end{proof}

\section{Properties of Lyapunov-optimal orbits}\label{s:geometry}

In this section we explore consequences of \autoref{t:Barabanov}.
Let us remark that the results of this section do not require the nonoverlapping condition.

Fix a dominated one-step cocycle with generator $(A_1, \dots, A_k)$,
a strictly forward-invariant multicone $\multicone$,
and Barabanov functions $p^\top$, $p^\bot$ on $\multicone$.

\subsection{Geometrical obstructions}

In this subsection, we will show that the invariant directions
of points on the Mather sets must obey certain geometrical obstructions.

We begin by considering sets of optimal future trajectories.
For each $\star \in \{\top, \bot\}$, let
$$
J^\star \coloneqq \big\{ (\omega_+,x) \in k^{\Z_+} \times \multicone; \;
p^\star (A^{(n)}(\omega_+) x) = p^\star(x) + n \lambda^\star_1 \ \forall n \ge 0 \big\} \, .
$$
Since the functions $p^\star$ are continuous, these sets are closed.

Notice that, as a consequence of properties \eqref{eq:ptop} and \eqref{eq:pbot} of the Barabanov functions, the following holds:
$$
\forall \, x \in \multicone \ \exists \, \omega_+ \in k^{\Z_+} \text{ such that } (\omega_+,x) \in J^\star \, .
$$

\begin{lemma} \label{l:ineq_barabanov}
If $(\omega_+,x) \in J^\star$ and $y \in \multicone$ are such that $x-y \in \tilde{e}_2(\omega_+)$ (see \autoref{f:xy})
then:
\begin{alignat*}{2}
p^\top (x) &\le p^\top (y) &\quad &\text{if $\star = \top$,}\\
p^\bot (x) &\ge p^\bot (y) &\quad &\text{if $\star = \bot$.}
\end{alignat*}
\end{lemma}

\begin{figure}[htb]
	\centering
	\begin{tikzpicture}[scale=2.3]
		\draw[thin]  (-.125,-.25) -- (.5,1) node[right] {\scriptsize $e_2(\omega)$};
		\draw[-stealth,thick]      (0,0) -- (1,0)   node[right]{\scriptsize $x \in M$ (optimal w.r.t.\ $\omega_+$)};
		\draw[dashed]  (.875,-.25) -- (1.5,1);
		\draw[-stealth,thick]      (0,0) -- (1.25,.5)  node[right]{\scriptsize $y\in M$};
	\end{tikzpicture}
	\caption{\autoref{l:ineq_barabanov}}
	\label{f:xy}
\end{figure}

\begin{proof}
Let $\omega_+ \in k^{\Z_+}$  and  $x$, $y \in \multicone$ be such that $x-y \in \tilde{e}_2(\omega_+)$.
Let $x_n \coloneqq A^{(n)}(\omega_+)x$ and $y_n \coloneqq A^{(n)}(\omega_+)y$, for $n \ge 0$.
Since $x$, $y \not\in \tilde{e}_2(\omega_+)$,
it follows from \autoref{l:dom_easy}
that the quantities
\begin{equation}\label{eq:vanishing}
\frac{\|x_n - y_n\|}{\|x_n\|} , \quad
\frac{\|x_n - y_n\|}{\|y_n\|} , \quad \text{and} \quad
\angle(x_n,y_n) \quad
\text{tend to $0$ as $n \to \infty$.}
\end{equation}

Let us now show that
\begin{equation}\label{eq:more_vanishing}
\lim_{n \to \infty} \big[ p^\star(y_n)  - p^\star(x_n) \big] = 0 \, .
\end{equation}
Indeed, by property \eqref{eq:p_Lip} of Barabanov functions,
$$
\big| p^\star(y_n)  - p^\star(x_n) \big| \le
c_0 \angle(y_n, x_n) + \big| \log\|y_n\| - \log\|x_n\| \big| \, .
$$
The first term tends to zero as $n \to \infty$.
The second term can be estimated as:
$$
\big| \log\|y_n\| - \log\|x_n\| \big|
\le \max \left( \frac{\|y_n\|}{\|x_n\|}-1, \frac{\|x_n\|}{\|y_n\|}-1 \right)
\le \frac{\|x_n - y_n\|}{\min \left( \|x_n\| , \|y_n\| \right)} ,
$$
which by \eqref{eq:vanishing} tends to zero as well.
This proves \eqref{eq:more_vanishing}.

Next, assume $(\omega_+,x) \in J^\star$.
So, for all $n \ge 0$,
\[
p^\star(x_n) = p^\star(x) + n\lambda^\star_1 \, .
\]
By properties \eqref{eq:ptop} and \eqref{eq:pbot} we have:
\begin{alignat*}{2}
p^\top(y_n) &\le p^\top(y) + n\lambda^\top_1 &\quad &\text{if $\star = \top$,} \\
p^\bot(y_n) &\ge p^\bot(y) + n\lambda^\bot_1 &\quad &\text{if $\star = \bot$.}
\end{alignat*}
In particular,
\begin{alignat*}{2}
p^\top(y_n)  - p^\top(x_n) &\le p^\top(y) - p^\top(x) &\quad &\text{if $\star = \top$,} \\
p^\bot(y_n)  - p^\bot(x_n) &\ge p^\bot(y) - p^\bot(x) &\quad &\text{if $\star = \bot$.}
\end{alignat*}
Taking limits as $n\to \infty$ and recalling \eqref{eq:more_vanishing}
we obtain the lemma.
\end{proof}

Given vectors $x_1$, $y_1$, $x_2$, $y_2\in \R^2_*$, no three of them collinear,
we define their \emph{cross-ratio}
\[
[x_1, y_1; x_2, y_2] \coloneqq \frac{x_1\times x_2}{x_1 \times y_2} \cdot \frac{y_1\times y_2}{y_1 \times x_2} \in \R \cup \{\infty\} \, ,
\]
where $\times$ denotes cross-product in $\R^2$, i.e.\ determinant. See \cite[Section I.6]{BK}.
The cross-ratio actually depends only on the directions defined by the four vectors, which allows us to apply the same definition to $4$-tuples in $(\P^1)^4$
without three coinciding points.
Moreover, the cross-ratio is invariant under linear transformations.

\smallskip

We now use \autoref{l:ineq_barabanov} to prove the following important \autoref{l:ineq_cross_ratio_J},
whose character is similar to Proposition~2.6 from \cite{BM}:

\begin{lemma} \label{l:ineq_cross_ratio_J}
For all $(\xi,x_1)$, $(\eta,y_1) \in J^\star$
and nonzero vectors $x_2\in \tilde{e}_2(\xi)$, $y_2\in \tilde{e}_2(\eta)$ we have
\begin{alignat*}{2}
|[x_1, y_1; x_2, y_2]| &\ge 1 &\quad &\text{if $\star = \top$,} \\
|[x_1, y_1; x_2, y_2]| &\le 1 &\quad &\text{if $\star = \bot$.}
\end{alignat*}
\end{lemma}

\begin{proof}
Let us consider the case of $J^\top$; the other case is analogous.

Recall from \autoref{p:nested} that any $e_1$ direction is different from any $e_2$ direction.
So, neither $x_1$ nor $y_1$ can be collinear to $x_2$ or $y_2$.
Hence the cross-ratio is well defined.
Moreover, we can write:
$$
x_1 = \alpha x_2 + \beta y_1 \quad\text{and}\quad
y_1 = \gamma y_2 + \delta x_1.
$$
By \autoref{l:ineq_barabanov},
\[
p^\top (x_1) \leq p^\top (\beta y_1) \leq p^\top (\beta \delta x_1) = p^\top (x_1) + \log |\beta\delta|.
\]
Hence, $|\beta \delta| \geq 1$.
Substituting
$$
\beta  = \frac{x_1\times x_2}{y_1 \times x_2} \quad\text{and}\quad
\delta = \frac{y_1 \times y_2}{x_1 \times y_2}
$$
we obtain the assertion.
\end{proof}

The sets $J^\star$ are related with the Mather sets $K^\star$.
Indeed, as a consequence of the inclusion \eqref{eq:Mather_included}
and the definitions \eqref{eq:def_L} and \eqref{eq:def_psi}, we have:
$$ 
\left.
\begin{array}{l}
\omega \in K^\star \\
x \in e_1(\omega)\setminus \{0\}
\end{array}
\right\}
\ \Rightarrow \
(\pi_+(\omega),x) \in J^\star \, .
$$ 
where $\pi_+$ is the projection defined by \eqref{eq:pi+}.
Therefore we immediately obtain
the following consequence of \autoref{l:ineq_cross_ratio_J}:

\begin{lemma}\label{l:ineq_cross_ratio_K}
If $\xi$, $\eta \in K^\star$ then
\begin{alignat*}{2}
\big| [e_1(\xi), e_1(\eta) ; e_2(\xi), e_2(\eta)] \big| &\ge 1 &\quad &\text{if $\star = \top$,} \\
\big| [e_1(\xi), e_1(\eta) ; e_2(\xi), e_2(\eta)] \big| &\le 1 &\quad &\text{if $\star = \bot$.}
\end{alignat*}
\end{lemma}

In the rest of the section we will use the hyperbolic geometry representation of the projective space $\P^1$. Consider the
unit disk $\D \coloneqq \{ z \in \C ; \; |z| < 1 \}$ endowed with the Poincar\'e hyperbolic metric.
Given two different points $x_1$, $x_2$ in the unit circle $\partial \D$, let
$\overrightarrow{x_2 x_1}$ denote the oriented hyperbolic geodesic from $x_2$ to $x_1$.
We identify $\partial \D$ with the projective space $\P^1$ as follows:
$$
e^{2 \theta i} \in \partial \D \ \leftrightarrow \
\project{(\cos \theta , \sin \theta)} \in \P^1 \, .
$$

Let $(x_1, y_1 ; x_2, y_2)$ be a $4$-tuple of distinct points in $\P^1$.
Then one and only one of the following possibilities holds:
\begin{itemize}
\item \emph{antiparallel configuration}: $x_1 < y_2 < y_1 < x_2 < x_1$ for some cyclic order $<$ on $\P^1$ (see \autoref{f:antipar});
\item \emph{coparallel configuration}:   $x_1 < y_1 < y_2 < x_2 < x_1$ for some cyclic order $<$ on $\P^1$ (see \autoref{f:copar});
\item \emph{crossing configuration}:     $x_1 < y_1 < x_2 < y_2 < x_1$ for some cyclic order $<$ on $\P^1$ (see \autoref{f:cross}).
\end{itemize}
We say that two geodesics
$\overrightarrow{x_2 x_1}$ and $\overrightarrow{y_2 y_1}$ with distinct endpoints
are \emph{antiparallel}, \emph{coparallel}, or \emph{crossing}
according to the configuration of the $4$-tuple $(x_1,y_1;x_2,y_2)$.

\begin{figure}[htb]
	\centering
	\begin{minipage}{0.3\textwidth}
		\centering
		\begin{tikzpicture}[scale=1.2]
			\draw (0,0) circle [radius=1];
			\node[above right] at ( .7071, .7071) {\scriptsize $y_2$};
			\node[above left]  at (-.7071, .7071) {\scriptsize $y_1$};
			\node[below left]  at (-.7071,-.7071) {\scriptsize $x_2$};
			\node[below right] at ( .7071,-.7071) {\scriptsize $x_1$};
			\draw[thick, postaction={decorate,decoration={markings, mark=at position .53 with {\arrow{stealth}}}}] (.7071, .7071) arc [radius=1, start angle=315, end angle=225];
			\draw[thick, postaction={decorate,decoration={markings, mark=at position .53 with {\arrow{stealth}}}}] (-.7071,-.7071) arc [radius=1, start angle=135, end angle= 45];
		\end{tikzpicture}
		\caption{Antiparallel configuration}
		\label{f:antipar}
	\end{minipage}
	\hfill
	\begin{minipage}{0.3\textwidth}
		\centering
		\begin{tikzpicture}[scale=1.2]
			\draw (0,0) circle [radius=1];
			\node[above right] at ( .7071, .7071) {\scriptsize $y_1$};
			\node[above left]  at (-.7071, .7071) {\scriptsize $y_2$};
			\node[below left]  at (-.7071,-.7071) {\scriptsize $x_2$};
			\node[below right] at ( .7071,-.7071) {\scriptsize $x_1$};
			\draw[thick, postaction={decorate,decoration={markings, mark=at position .53 with {\arrow{stealth}}}}] (-.7071, .7071) arc [radius=1, start angle=225, end angle=315];
			\draw[thick, postaction={decorate,decoration={markings, mark=at position .53 with {\arrow{stealth}}}}] (-.7071,-.7071) arc [radius=1, start angle=135, end angle= 45];	
		\end{tikzpicture}
		\caption{Coparallel configuration}
		\label{f:copar}
	\end{minipage}
	\hfill
	\begin{minipage}{0.3\textwidth}
		\centering
		\begin{tikzpicture}[scale=1.2]
			\draw (0,0) circle [radius=1];
			\node[above right] at ( .7071, .7071) {\scriptsize $y_1$};
			\node[above left]  at (-.7071, .7071) {\scriptsize $x_2$};
			\node[below left]  at (-.7071,-.7071) {\scriptsize $y_2$};
			\node[below right] at ( .7071,-.7071) {\scriptsize $x_1$};
			\draw[thick, postaction={decorate,decoration={markings, mark=at position .75 with {\arrow{stealth}}}}] (-.7071,-.7071) -- ( .7071, .7071);
			\draw[thick, postaction={decorate,decoration={markings, mark=at position .75 with {\arrow{stealth}}}}] (-.7071, .7071) -- ( .7071,-.7071);
		\end{tikzpicture}
		\caption{Crossing configuration}
		\label{f:cross}
	\end{minipage}
\end{figure}

The configuration is expressed in terms of the cross-ratio as follows:

\begin{proposition}\label{p:cross_ratio_cases}
Consider a $4$-tuple $(x_1, y_1 ; x_2, y_2)$ of distinct points in $\P^1$.
Then:
\begin{itemize}
\item the configuration is \emph{antiparallel} iff $[x_1,y_1;x_2,y_2] < 0$;
\item the configuration is \emph{coparallel}   iff $0 < [x_1,y_1;x_2,y_2] < 1$;
\item the configuration is \emph{crossing}     iff $[x_1,y_1;x_2,y_2] > 1$.
\end{itemize}
\end{proposition}

\begin{proof}
With a linear change of coordinates, we can assume that the directions
$y_1$, $x_2$, $y_2$
contain the vectors $(1,1)$, $(1,0)$, $(0,1)$, respectively.
Let $(a,b)$ be a nonzero vector in the $x_1$ direction.
Then $[x_1,y_1;x_2,y_2] = b/a$.
The proposition follows by inspection.
\end{proof}

Define the following compact subsets of the torus $\P^1 \times \P^1$:
\begin{equation}\label{eq:G}
G^\star \coloneqq \big\{ (e_1(\omega), e_2(\omega)) ; \; \omega \in K^\star \big\} \, ,
\end{equation}
where $\star$ is either $\top$ or $\bot$.
The possibilities for those sets are severely restricted; indeed:

\begin{corollary}\label{c:forbid}
Let $(x_1,x_2)$, $(y_1,y_2) \in G^\star$.
Then:
\begin{itemize}
\item if $\star=\top$ then $(x_1, y_1 ; x_2, y_2)$ cannot be in coparallel configuration;
\item if $\star=\bot$ then $(x_1, y_1 ; x_2, y_2)$ cannot be in crossing configuration.
\end{itemize}
\end{corollary}

\begin{proof}
We can assume that $x_1$, $x_2$, $y_1$, $y_2$ are distinct, otherwise there is nothing to prove.
The claims follow by combining \autoref{l:ineq_cross_ratio_K} and \autoref{p:cross_ratio_cases}.
\end{proof}

\subsection{Each invariant direction essentially determines the other}

Now we will show that for points $\omega$ on the Mather sets,
each invariant direction $e_1(\omega)$ or $e_2(\omega)$
uniquely determines the other,
except for a countable number of bad directions.
This fact (stated precisely in \autoref{l:countable} below)
is actually a simple consequence of \autoref{c:forbid},
and forms the core of the proof of \autoref{t:zero}.

\medskip

Consider the set $G^\star$ defined by \eqref{eq:G};
we decompose it into fibers in two different ways:
$$
G^\star
= \bigcup_{x_1 \in e_1(K^\star)} \{x_1\} \times G^\star_2(x_1)
= \bigcup_{x_2 \in e_2(K^\star)} G^\star_1(x_2) \times \{x_2\} \, .
$$
Define also
\begin{align}
N^\star_1 &\coloneqq \{ x_1  \in e_1(K^\star) ; \; G^\star_2(x_1) \text{ has more than one element} \} \, , \label{eq:N1}\\
N^\star_2 &\coloneqq \{ x_2  \in e_2(K^\star) ; \; G^\star_1(x_2) \text{ has more than one element} \} \, . \label{eq:N2}
\end{align}
So the following implication holds:
\begin{equation}\label{eq:N_implication}
\left.
\begin{array}{l}
\xi, \eta \in K^\star  \\
e_i(\xi) = e_i(\eta) \notin N^\star_i \text{ for some $i$}
\end{array}
\right\}
\ \Rightarrow \
\left\{
\begin{array}{l}
e_1(\xi) = e_1(\eta) \\
e_2(\xi) = e_2(\eta)
\end{array}
\right.
\end{equation}

\begin{lemma}\label{l:countable}
For each $\star \in \{\top, \bot\}$ and $i \in \{1,2\}$,
the set $N^\star_i$ is countable.
\end{lemma}

\begin{proof}
We will consider the case $i=1$; the case $i=2$ is entirely analogous.

For each $x \in N^\star_1$, let
$I^\star(x)$ be the least closed subinterval of $\P^1 \setminus \{x\}$
containing $G^\star_2(x)$.

We begin with the case of $N^\top_1$.

\begin{claim}\label{cl:disjoint_intervals}
If $x$, $y \in N^\top_1$ are distinct
then $I^\top(x)$ and $I^\top(y)$ have disjoint interiors in the circle $\P^1$.
(See \autoref{f:intervals}.)
\end{claim}

\begin{figure}[htb]
	
	\centering
	\begin{minipage}{0.4\textwidth}
		\centering
		\begin{tikzpicture}[scale = 2]
			\draw (0,0) circle [radius=1];
					
			\draw[postaction={thick,decorate,decoration={markings, mark=at position .75 with {\arrow{stealth reversed}}}}] (cos 50, sin 50) node[above right]{\footnotesize $y$}
			arc[radius=tan 75, start angle=320, end angle=290]; 
			\draw[postaction={thick,decorate,decoration={markings, mark=at position .75 with {\arrow{stealth reversed}}}}] (cos 50, sin 50)
			arc[radius=tan 85, start angle=140, end angle=150]; 
			\draw[ultra thick] (cos 200, sin 200)
			arc[radius=1, start angle=200, end angle=240];
			\draw (cos 220, sin 220) node[below left]{\footnotesize $I^\top(y)$};

			\draw[postaction={thick,decorate,decoration={markings, mark=at position .75 with {\arrow{stealth reversed}}}}] (cos 20, -sin 20) node[right]{\footnotesize $x$}
			arc[radius=tan 50, start angle=250, end angle=170];
			\draw[postaction={thick,decorate,decoration={markings, mark=at position .75 with {\arrow{stealth reversed}}}}] (cos 20, -sin 20)
			arc[radius=tan 80, start angle=250, end angle=230];
			\draw[ultra thick] (cos 80, sin 80)
			arc[radius=1, start angle=80, end angle=140];
			\draw (cos 110, sin 110) node[above left]{\footnotesize $I^\top(x)$};

		\end{tikzpicture}
		\caption{$x \neq y \in N^\top_1$; the intervals $I^\top(x)$ and $I^\top(y)$ have disjoint interiors.}
		\label{f:intervals}
	\end{minipage}
	\hfill
	\begin{minipage}{0.4\textwidth}
		\centering
		\begin{tikzpicture}[scale = 2]
			\draw (0,0) circle [radius=1];
									
			\filldraw[pattern=horizontal lines light gray] (cos 30, sin 30)
			arc[radius=tan 60, start angle=300, end angle=240]
			arc[radius=tan 70, start angle=60, end angle=20] node[below]{\footnotesize $x$}
			arc[radius=tan 50, start angle=200, end angle=120];
			\draw (.09, .03) node{\footnotesize $\Delta(x)$};
			\draw[postaction={thick,decorate,decoration={markings, mark=at position .53 with {\arrow{stealth}}}}] (cos 150, sin 150) arc[radius=tan 70, start angle=60, end angle=20];
			\draw[postaction={thick,decorate,decoration={markings, mark=at position .53 with {\arrow{stealth}}}}] (cos 30, sin 30) arc[radius=tan 50, start angle=120, end angle=200];

			\filldraw[pattern=horizontal lines light gray] (cos 50, sin 50) node[above right]{\footnotesize $y$}
			arc[radius=tan 20, start angle=320, end angle=180]
			arc[radius=tan 25, start angle=360, end angle=230]
			arc[radius=tan 45, start angle=230, end angle=320];
			\draw (.03, .56) node{\footnotesize $\Delta(y)$};
			\draw[postaction={thick,decorate,decoration={markings, mark=at position .47 with {\arrow{stealth reversed}}}}] (cos 50, sin 50) arc[radius=tan 20, start angle=320, end angle=180];
			\draw[postaction={thick,decorate,decoration={markings, mark=at position .47 with {\arrow{stealth reversed}}}}] (cos 50, sin 50) arc[radius=tan 45, start angle=320, end angle=230];

		\end{tikzpicture}
		\caption{$x \neq y \in N^\bot_1$; the triangles $\Delta(x)$ and $\Delta(y)$ have disjoint interiors.}
		\label{f:triangles}
	\end{minipage}
\end{figure}

\begin{proof}[Proof of the claim]
Let $v$ and $w$ be the endpoints of the interval $I^\top(x)$
and take any point $z$ in its interior.
Then the geodesic $\overrightarrow{z y}$
is coparallel to one of the two geodesics
$\overrightarrow{v x}$ or $\overrightarrow{w x}$.
Since $(x,v)$ and $(x,w)$ belong to $G^\top$,
by \autoref{c:forbid} we conclude that $(y,z)$ does not.
This shows that $G^\top_2(y) \cap \Int I^\top(x) = \emptyset$,
and, in particular, $\partial I^\top(y) \cap \Int I^\top(x) = \emptyset$.
An analogous argument gives $\partial I^\top(x) \cap \Int I^\top(y) = \emptyset$.
It follows that $\Int I^\top(x) \cap \Int I^\top(y) = \emptyset$.
\end{proof}

It follows from separability of the circle that $N^\top_1$ is countable.

\medskip

Now let us consider the case of $N^\bot_1$.
For each $x \in N^\bot_1$, let $\Delta(x)$ be the ideal triangle
whose vertices are $x$ and the two endpoints of the interval $I^\bot(x)$.

\begin{claim}\label{cl:disjoint_triangles}
If $x$, $y \in N^\bot_1$ are distinct
then $\Delta(x)$ and $\Delta(y)$ have disjoint interiors in the disk $\D$.
(See \autoref{f:triangles}.)
\end{claim}

\begin{proof}[Proof of the claim]
Let $v$ and $w$ be the endpoints of the interval $I^\top(x)$.
Since these points belong to $e_2(K^\bot)$, which is disjoint from $e_1(K^\bot)$,
none of them can be equal to $y$.
Let $C$ be the connected component of $\D \setminus \Int \Delta(x)$
whose closure at infinity contains $y$.
Let $z \in G^\bot_2(y)$.
By \autoref{c:forbid}, the geodesic $\overrightarrow{z y}$
does not cross $\overrightarrow{v x}$ nor $\overrightarrow{w x}$.
It follows that $\overrightarrow{z y}$ is disjoint from $\Int \Delta(x)$,
and so it is contained in $C$.
Since $C$ is geodesically convex, it follows that $\Delta(y) \subset C$.
This proves the claim.
\end{proof}

It follows from separability of the disc that $N^\bot_1$ is also countable,
thus completing the proof of \autoref{l:countable}.
\end{proof}

\section{Obtaining zero entropy}\label{s:zero}

In this section we conclude the proof of \autoref{t:zero}.

Fix a dominated one-step cocycle with generator $(A_1, \dots, A_k)$,
and fix $\star \in \{\top,\bot\}$.
Recall the definition \eqref{eq:N1} of the set $N_1^\star$.

\begin{lemma}\label{l:g}
There exists a Borel measurable map
$g^\star \colon e_1(k^\Z) \to e_2(k^\Z)$
such that
\begin{equation}\label{eq:g_property}
\left.
\begin{array}{l}
\omega \in K^\star \\
e_1(\omega) \notin N^\star_1
\end{array}
\right\}
\ \Rightarrow \
g^\star(e_1(\omega)) = e_2(\omega) \, .
\end{equation}
\end{lemma}

\begin{proof}
Let $x_1 \in e_1(k^\Z)$.
If $x_1 \in e_1(K^\star) \setminus N_1^\star$ then
there exists a unique $x_2 \in e_2(K^\star)$
such that $(x_1,x_2) \in G^\star$;
let $g^\star(x_1) \coloneqq x_2$.
Otherwise if $x_1 \not\in e_1(K^\star) \setminus N_1^\star$ then
let $g^\star(x_1) \coloneqq p$,
where $p \in e_2(k^\Z)$ is an arbitrary constant.
Thus we have defined a map $g^\star$ that satisfies property~\eqref{eq:g_property}.
Its graph is
$$
\big[ G^\star \setminus (N_1^\star \times \P^1) \big]
\, \cup \,
\big[ \big( e_1(k^\Z) \setminus (e_1(K^\star)\setminus N_1^\star) \big) \times \{p\} \big]  \, ,
$$
and therefore is a Borel measurable subset of $e_1(k^\Z) \times e_2(k^\Z)$.
It follows (use \cite[Lemma~6.7.1]{Bogachev}) that $g^\star$ is a Borel measurable map,
as we wanted to show.
\end{proof}

By the entropy variational principle (see \cite[p.~269]{Petersen}),
in order to prove that the restriction of $T$ to the compact invariant set $K^\star$
has zero topological entropy,
it is sufficient to prove that $h_\mu(T) = 0$ for every ergodic probability measure
$\mu$ supported on $K^\star$.
Fix any such measure $\mu$.
Let us assume that $\mu$ is non-atomic, because otherwise there is nothing to prove.

\begin{lemma}\label{l:0_meas}
$\mu\left(e_1^{-1}(N^\star_1)\right) = 0$.
\end{lemma}

\begin{proof}
By \autoref{l:countable}, the set $N^\star_1 \subset \P^1$ is countable.
Since $\sA$ has the forward NOC, it follows from \autoref{p:NOC} that the set $e_1^{-1}(N^\star_1) \subset k^\Z$
is a countable union of sets of the form $\{\omega_-\} \times k^{\Z_+}$.
Assume for a contradiction that $e_1^{-1}(N^\star_1)$ has positive measure.
Then there exists $\omega_- \in k^{\Z_-}$ such that
$F \coloneqq \{\omega_-\} \times k^{\Z_+}$ has positive measure.
By Poincar\'e recurrence, there exists $p \ge 1$ such that
$T^{-p}(F) \cap F \neq \emptyset$.
It follows that the infinite word $\omega_-$ is periodic with period $p$,
which in turn implies that $T^{-p}(F) \subset F$.
By invariance, $\mu \left(F \setminus T^{-p}(F) \right) = 0$
and
\begin{align*}
\mu \left( \bigcap_{n\ge 0} T^{-np}(F) \right)
&= \mu(F) - \mu \left(F \setminus T^{-p}(F) \right) - \mu \left(T^{-p}(F) \setminus T^{-2p}(F) \right) - \cdots \\
&= \mu(F) > 0.
\end{align*}
But the set $\bigcap_{n \ge 0} T^{-np}(F)$ is a singleton,
thus contradicting the assumption that $\mu$ is non-atomic.
This proves the lemma.
\end{proof}

As an immediate consequence of the previous two lemmas, we obtain:
\begin{equation}\label{eq:e2_determines_e1}
e_2 = g^\star \circ e_1 \quad \mu\text{-a.e.}
\end{equation}
That is, \emph{the direction $e_1$ almost surely determines $e_2$}.

Consider the continuous maps $\tilde{e_1}$, $\tilde{e_2}$ given by \autoref{c:factor_projections}.
Due to the backwards NOC,
the map $\tilde{e_2}$ is one-to-one (\autoref{p:NOC}), and therefore there exists
a unique map $f^\star$ that makes the following diagram commutative:
$$
\begin{tikzcd}[column sep=large]
k^{\Z_-}  \arrow{d}[swap]{\tilde{e}_1} \arrow{r}{f^\star}
& k^{\Z_+} \arrow{d}{\tilde{e}_2}  \\
e_1(k^\Z) \arrow{r}[swap]{g^\star} & e_2(k^\Z)
\end{tikzcd}
$$
Moreover, the map $f^\star$ is Borel measurable, and
by the commutativity relations \eqref{eq:diagrams} and
\eqref{eq:e2_determines_e1}, it satisfies:
\begin{equation}\label{eq:back_to_the_future}
\pi_+ = f^\star \circ \pi_- \quad \mu\text{-a.e.}
\end{equation}
That is, \emph{the past almost surely determines the future}.
It is known that this property implies zero entropy,
but for the reader's convenience let us spell out the details.

Let $\cC = \{C_1, \dots, C_k\}$ be the partition of $k^\Z$ into cylinders
$$
C_j \coloneqq \{\omega \in k^\Z ;\; \omega_0 = j\}.
$$
Define another partition
$\tilde \cC = \{\tilde{C}_1, \dots, \tilde{C}_k\}$
by
$$
\tilde{C}_j \coloneqq \pi_-^{-1} \left( (f^\star)^{-1} \left(\pi_+(C_j)\right) \right) \, .
$$
These are Borel sets that actually belong to the
the following $\sigma$-algebra:
$$
\cC_{-\infty}^{-1}
\coloneqq \bigvee_{n<0} T^{-n}(\cC) \, .
$$
By \eqref{eq:back_to_the_future},
$\mu(C_j \vartriangle \tilde{C}_j) = 0$,
that is $\cC = \tilde \cC$ modulo zero sets.
So, using the Kolmogorov--Sinai theorem
and other basic facts about entropy (see e.g.\ \cite{Petersen}),
we obtain:
\begin{alignat*}{2}
h_\mu(T) &= h_\mu(\cC, T)
&\quad &\text{(since $\cC$ is a generating partition)}
\\ &= h_\mu(\cC, T^{-1})
\\ &= H_\mu(\cC | \cC_{-\infty}^{-1})
\\ &= H_\mu(\tilde \cC | \cC_{-\infty}^{-1})
&\quad &\text{(since $\cC = \tilde \cC$ modulo zero sets)}
\\ &= 0
&\quad &\text{(since $\tilde \cC \subset \cC_{-\infty}^{-1}$.)}
\end{alignat*}
This proves \autoref{t:zero}.

\section{Obtaining positive entropy}\label{s:positive}

In this section we prove \autoref{t:positive}.

\subsection{Sufficient conditions for the existence of many bounded products}

\begin{lemma}\label{l:bounded_products}
Given a sequence $B_0$, $B_1$, \dots of matrices in $\SL(2,\R)$,
let $P_i \coloneqq B_{i-1} \cdots B_0$
and let $u_i$, $v_i$ be unit vectors in $\R^2$ such that $P_i u_i = \|P_i\| v_i$.
Suppose that there are constants $0 < \kappa < 1 < C$
such that
$$
\|B_i \| \le C \quad \text{and} \quad \|B_i v_i\| \le \kappa
\quad \text{for every $i$.}
$$
Then
$$
\|P_i \| \le \frac{\sqrt{2} \, C}{\sqrt{1-\kappa^2}}
\quad \text{for every $i$.}
$$
\end{lemma}

\begin{proof}
Recall that the Hilbert--Schmidt norm of a matrix $A$ is defined as
$\|A\|_\mathrm{HS} \coloneqq \sqrt{\tr A^* A}$.
If $A \in \SL(2,\R)$ then $\|A\|_\mathrm{HS}^2 = \|A\|^2 + \|A\|^{-2}$.

Let $B_i$, $P_i$, $u_i$, $v_i$, $C$ and $\kappa$ be as in the statement of the lemma.
Let $v^\perp_i$ be a unit vector orthogonal to $v_i$.
With respect to the basis $\{v_i, v_i^\perp\}$ we can write
$$
P_i P_i^* = \begin{pmatrix} \rho_i^2 & 0 \\ 0 & \rho_i^{-2} \end{pmatrix} \quad \text{and} \quad
B_i^* B_i = \begin{pmatrix} \alpha_i & \beta_i \\ \beta_i & \gamma_i \end{pmatrix} \, ,
$$
where $\rho_i = \|P_i\|$ and $\alpha_i = \langle B_i^* B_i v_i , v_i \rangle = \|B_i v_i\|^2$.
So
\begin{align*}
\|P_{i+1}\|_\mathrm{HS}^2
&=   \tr B_i^* B_i P_i P_i^*  \\
&=   \alpha_i \rho_i^2 + \gamma_i \rho_i^{-2} \\
&=   \|B_i v_i\|^2  \|P_i\|^2 + \left( \|B_i\|_\mathrm{HS}^2 - \|B_i v_i\|^2 \right) \|P_i\|^{-2} \\
&\le \|B_i v_i\|^2  \|P_i\|_\mathrm{HS}^2 + \|B_i\|_\mathrm{HS}^2 \\
&\le \kappa^2  \|P_i\|_\mathrm{HS}^2 + 2 C^2 \, .
\end{align*}
It follows by induction that $\|P_i\|_\mathrm{HS}^2 \le 2C^2/(1-\kappa^2)$ for every $i$,
which implies the lemma.
\end{proof}

Given $\sA = (A_1,\dots,A_k) \in \SL(2,\R)^k$, let $\langle \sA \rangle$ be
the semigroup generated by $\sA$, that is, the set of
all products of the form $A_{i_n}\dots A_{i_1}$ (where $n \ge 1$).

Let $\cC$ be the set of $\sA \in \SL(2,\R)^k$ such that
for every $v \in \R^2$ and $\epsilon > 0$
there exists $P \in \langle \sA \rangle$
such that $\| P v \| < \epsilon$.
It is easily seen that $\sA \in \cC$ if and only if
for every unit vector $v \in S^1$
there exists $P = A_{i_{n(v)}} \cdots A_{i_1} \in \langle \sA \rangle$
such that $\| P v \| < 1/2$.
It follows from compactness of the unit circle
that the lengths $n(v)$ can be chosen uniformly bounded,
and that $\cC$ is open.

\begin{lemma}\label{l:real_job}
Every $\sA \in \cC$ satisfies the second alternative in \autoref{t:positive}.
\end{lemma}

\begin{proof}
Fix $\sA \in \cC$. We extend the notation $A^{(n)}(\omega)$ to negative $n$'s by taking
\[
A^{(-n)}(\omega) \coloneqq \big[A^{(n)}(T^{-n}\omega)\big]^{-1} \, .
\]
Let $C \coloneqq \max \|A_i\|$.
It is an easy exercise to show that
there exist $\kappa \in (0,1)$ and an integer $\ell \ge 2$
such that for every unit vector $v \in \R^2$ there exists a product
$P \in \langle \sA \rangle$ of length $\ell-1$ such that
$\|P v \| < C^{-1}\kappa$,
and in particular
$\|A_i P v \| < \kappa$ for every $i=1,\dots,k$.
Let
$$
L \coloneqq \big\{ \omega \in k^\Z ; \; \|A^{(\ell n)}(\omega)\| \le C_1 \ \forall n \in \Z \big\} \, ,
\quad \text{where} \quad
C_1 \coloneqq \frac{\sqrt{2} \, C^{\ell}}{\sqrt{1-\kappa^2}}
$$
Then $L \neq \emptyset$; indeed, in order to find a bi-infinite word $\omega$ in $L$, we construct its future $\pi_+(\omega)$ and its past $\pi_-(\omega)$ separately, using \autoref{l:bounded_products} in a recursive way.
Moreover, given any bi-infinite sequence of
symbols $\dots, \omega_{-\ell}, \omega_0, \omega_{\ell}, \omega_{2\ell},\dots$
in the alphabet $\{1,\dots,k\}$, we can choose the remaining symbols to form a word $\omega$ in $L$.

Let
$$
K \coloneqq \big\{ \omega \in k^\Z ; \; \|A^{(n)}( T^m \omega)\| \le C^{2\ell} C_1^2  \ \forall n,m \in \Z \big\} \, .
$$
By definition, this set is compact and $T$-invariant,
and it is easy to see that it contains $L$.
It follows from the previous observations about $L$ that
the topological entropy of $K$ is at least $\ell^{-1} \log k$,
and thus positive as required.
\end{proof}

\subsection{Checking denseness}

Let $\cH$ be the set of $\sA \in \SL(2,\R)^k$ such that the one-step
cocycle generated by $\sA$ is uniformly hyperbolic.
Consider the open set $\cU \coloneqq \cH \cup \cC$.
By \autoref{l:real_job}, every element of $\cU$
satisfies one of the alternatives of \autoref{t:positive}.
Therefore, to prove the theorem, it is sufficient to show that $\cU$ is dense.

Let $\cE$ be the set of $\sA \in \SL(2,\R)^k$ for which the semigroup
$\langle \sA \rangle$ contains an elliptic element $R$
(that is, such that  $\left| \tr R \right| < 2$).
The sets $\cH$ and $\cE$ are open and pairwise disjoint.
We recall the following result:

\begin{theorem}[{\cite[Prop.~6]{Yoccoz}}]\label{t:UH_or_elliptic}
$\cH \cup \cE$ is dense in $\SL(2,\R)^k$.
\end{theorem}

Therefore, to show that $\cU \coloneqq \cH \cup \cC$ is dense in $\SL(2,\R)^k$,
we need to show:

\begin{lemma}\label{l:elliptic_helps}
$\cC \cap \cE$ is dense in $\cE$.
\end{lemma}

Let $\cI$ be the set of $\sA \in \cE$ such that
$\langle \sA \rangle$ contains a matrix conjugate to an irrational rotation.

\begin{lemma}\label{l:irrational}
$\cI$ is dense in $\cE$.
\end{lemma}

\begin{proof}
Let $(A_1,\dots,A_k) \in \cE$,
and fix an elliptic product $A_{i_n} \dots A_{i_1}$.
Let $P_\theta \coloneqq R_\theta A_{i_n} \dots R_\theta A_{i_1}$,
where $R_\theta$ denotes the rotation by angle $\theta$.
By \cite[Lemma~A.4]{ABY},
the function $\theta \mapsto \tr P_\theta$
has a nonzero derivative at $\theta=0$.
Therefore we can find $\theta_0$ arbitrarily close to $0$
such that $P_{\theta_0}$ is conjugate to an irrational rotation.
Therefore $(R_{\theta_0} A_1, \dots, R_{\theta_0} A_k) \in \cI$,
proving the lemma.
\end{proof}

\begin{proof}[Proof of \autoref{l:elliptic_helps}]
Let $\cN$ be the set of $\sA = (A_1,\dots,A_k) \in \SL(2,\R)^k$
such that not all $A_i$ commute; then $\cN$ is open and dense.
We will show that
\begin{equation}\label{eq:NIC}
\cN \cap \cI \subset \cC \, ,
\end{equation}
and so the desired result will follow from \autoref{l:irrational}.

Take $\sA = (A_1,\dots,A_k) \in \cN \cap \cI$.
Let $R \in \langle \sA \rangle$ be conjugate to an irrational rotation.
Since the sets $\cN$, $\cI$ and $\cC$ are invariant by conjugation,
we can assume that $R$ is an irrational rotation.
Since $\sA \in \cN$, there exists a generator $A_i$ that does not commute with $R$.
Using the singular value decomposition of $A_i$, we see that there exist
$n$, $m \ge 0$ such that $H \coloneqq R^n A_i R^m$ is a hyperbolic matrix.
Let $s$ be the contracting eigendirection of $H$.
Now, given any unit vector $v$ and any $\epsilon>0$,
we can find $j \ge 0$ such that the unit vector $R^j v$ is sufficiently close to $s$,
and so there exists $\ell \ge 0$ such that $\| H^\ell R^j v\| < \epsilon$.
This shows that $\sA \in \cC$,
thus proving \eqref{eq:NIC} and the lemma.
\end{proof}

As explained before, \autoref{t:positive} follows.

\medskip

Comparing to the present paper, the proof of Theorem~2 in \cite{BoBoDi} uses similar but slightly simpler arguments to get zero exponents. It does not obtain bounded norms, however. The present construction, especially \autoref{l:bounded_products}, is more related to strategy suggested on \cite[Remark~11.3]{BoBoDi}.

\appendix
\section{Complement}\label{s:appendix}

\subsection{Optimization of other dynamical quantities}

The results we have proved up to this point concern the optimization (maximization or minimization)
of the upper Lyapunov exponent $\lambda_1$.
Let us discuss briefly how to obtain results for the lower Lyapunov exponent $\lambda_2$
and for the difference $\lambda_1 - \lambda_2$ (which is a measure of non-conformality).

Suppose $T \colon X \to X$ is a continuous transformation of a
compact metric space and $A \colon X \to \GL(2,\R)$ is a continuous map.

\medskip

Define $B \colon X \to \GL(d,\R)$ by
\begin{equation}\label{eq:inverse}
B(\omega) \coloneqq A(T^{-1}\omega)^{-1},
\end{equation}
and consider it as a cocycle over $T^{-1}$.
Then a point $\omega\in X$ is Oseledets regular with respect to $(T,A)$
iff if it is regular with respect to $(T^{-1},B)$, and
$$
\lambda_1(T^{-1}, B, \omega) = -\lambda_2(T, A,\omega) \quad \text{and} \quad
\lambda_2(T^{-1}, B, \omega) = -\lambda_1(T, A,\omega) .
$$
In particular,
$$
\lambda^\top_2(T, A) = - \lambda^\bot_1 (T^{-1}, B)  \quad \text{and} \quad
\lambda^\bot_2(T, A) = - \lambda^\top_1 (T^{-1}, B) .
$$
If $(T,A)$ is an one-step cocycle then so is $(T^{-1},B)$
(after taking an appropriate conjugation between $T$ and $T^{-1}$),
and a multicone for one of them is a complementary multicone for the other.

It is then obvious how to adapt Theorems \ref{t:Mather}, \ref{t:zero} and \ref{t:positive}
to $\lambda_2$-optimization.

\medskip

Now define another matrix-valued map
\begin{equation}\label{eq:normalize}
C(\omega) \coloneqq \left|\det A(\omega)\right|^{-1/2} A(\omega).
\end{equation}
Then for all $\omega$ in a full probability set,
$$
\lambda_1(A,\omega) - \lambda_2(A,\omega)
= 2\lambda_1(C, \omega) = - 2\lambda_2(C,\omega) \, .
$$
Also note that the cocycle $(T,A)$ is dominated if and only if
$(T,C)$ is uniformly hyperbolic.
If $(T,A)$ is an one-step cocycle then so is $(T,C)$,
and a multicone for one of them is a multicone for the other.

It is then obvious how to adapt Theorems \ref{t:Mather} and \ref{t:zero}
to $(\lambda_1-\lambda_2)$-optimization.
In the converse direction, let us see $\SL(2,\R)$-cocycles,
can be adapted to cocycles taking values in
$\GL_+(2,\R)$ (the group of matrices with positive determinant)
as follows:

\begin{corollary}\label{c:positive}
Fix $k \ge 2$ and let $T$ be the full shift in $k$ symbols.
There exists an open and dense subset $\cV$ of $\GL_+(2,\R)^k$
such that for every $\sA \in \cV$,
\begin{enumerate}
\item either the one-step cocycle over $T$ generated by $\sA$ is dominated;
\item
or there exists a compact $T$-invariant set $K \subset k^\Z$
of positive topological entropy and such that the ``non-conformalities''
$\|A^{(n)}(\omega)\|/\fm(A^{(n)}(\omega))$.
are uniformly bounded over $(\omega,n) \in K \times \Z$.
\end{enumerate}
\end{corollary}

Notice that in the first case we have $(\lambda_1 - \lambda_2)^\bot(\sA) > 0$,
while in the second case there exists
a measure $\mu \in \cM_T$ such that
$(\lambda_1-\lambda_2)(\sA, \mu) = 0$ and moreover $h(T, \mu) > 0$.

\begin{proof}[Proof of \autoref{c:positive}]
Let $p \colon \GL_+(2,\R) \to \SL(2,\R)$
be the continuous open mapping $A \mapsto \left| \det A\right|^{-1/2} A$.
Let $\cU$ be given by \autoref{t:positive},
and define $\cV$ as the pre-image of $\cU$
by $p^k$ (the cartesian product of $k$ copies of $p$).
Then $\cV$ has the stated properties.
\end{proof}

\subsection{Alternative characterizations of extremal exponents and of domination}\label{ss:alternative}

Some remarks in the Introduction were left unjustified,
so let us deal with them now.

\medskip

First, equality \eqref{eq:alternative_top} actually holds in much greater generality,
and is related to the semiuniform subadditive ergodic theorem of Schreiber, Sturman, and Stark -- we refer the reader to \cite{Morris13} (see Theorems 2.1 and A.3) for a complete discussion.

\medskip

Next, let us prove relation \eqref{eq:alternative_bot}.
By subadditivity, its RHS equals
$$
R\coloneqq
\inf_n \frac{1}{n} \log \inf_{i_1, \dots, i_n} \| A_{i_n} \dots A_{i_1} \| \, .
$$
So $\lambda_1^\bot(\sA) \ge R$ by definition.
To check the converse inequality, fix $\epsilon>0$
and take symbols $i_1, \dots, i_n$ such that
$\frac{1}{n} \log \| A_{i_n} \dots A_{i_1} \| < R + \epsilon$.
Consider the shift-invariant probability measure on $k^\Z$ supported on the
periodic orbit $(i_1 \dots i_n)^\infty$.
Then $\lambda_1^\bot(\sA) \le \lambda_1(A,\mu) < R + \epsilon$.
Since $\epsilon$ is arbitrary, we conclude that $\lambda_1^\bot(\sA) = R$,
so proving \eqref{eq:alternative_bot}.

\medskip

Finally, let us show that an one-step $2\times 2$ cocycle is dominated
if and only if the number $(\lambda_1-\lambda_2)^\bot$
defined by \eqref{eq:diff_bot} is positive.
The ``only if'' part is evident, and actually does not require the one-step condition.
To prove the ``if'' part,
notice the equality
$$
(\lambda_1-\lambda_2)^\bot(\sA) =
\lim_{n\to \infty} \frac{1}{n} \log \inf_{i_1, \dots, i_n} \frac{\| A_{i_n} \dots A_{i_1} \|}{\fm(A_{i_n} \dots A_{i_1})} \, ,
$$
which follows from \eqref{eq:alternative_bot} applied
to the ``normalized'' one-step cocycle defined by \eqref{eq:normalize}.
So if this number is positive then
we can find positive constants $c$, $\delta$ such that \eqref{eq:domination_via_non_conformality}
holds, and therefore the cocycle is dominated.

Let us remark that for general cocycles, $(\lambda_1-\lambda_2)^\bot(A)>0$ does not imply
that the cocycle is dominated: for example $T$ can be uniquely ergodic and the cocycle can have different Lyapunov exponents without being dominated: see e.g.\ \cite[\S~4]{Herman}.

\subsection{More on the existence of optimizing measures}\label{ss:nonexistence}

Given a cocycle $(T,A)$,
the numbers $\lambda_1(A,\mu)$ and $\lambda_2(A,\mu)$ respectively
depend upper- and lower-semicontinuously on $\mu \in \cM_T$,
and therefore by compactness of $\cM_T$,
$\lambda_1$-maximizing and $\lambda_2$-minimizing measures always exist.
For a similar reason, $(\lambda_1-\lambda_2)$-maximizing measures always exist.

There are one-step cocycles where no $\lambda_1$-minimizing measure exists:
see \cite[Remark~1.2]{BMo}; a simple example is
$\sA = (H, c R_\theta)$ where $H\in \SL(2,\R)$ is hyperbolic, $\theta/\pi$ is irrational, and $c>1$.
Similarly, there are one-step cocycles where no $\lambda_2$-maximizing measure exists:
consider the same example with $c<1$ instead.

Let us give an example where no $(\lambda_1 - \lambda_2)$-minimizing measure exists.
We will actually exhibit an example of an one-step $\SL(2,\R)$-cocycle without
$\lambda_1$-minimizing measures.

Given a hyperbolic matrix $L$ in $\SL(2,\R)$, let $u_L$, $s_L \in \P^1$
denote its eigendirections, with $u_L$ corresponding to an eigenvalue of modulus bigger than $1$.
For convenience, the action of $L$ on $\P^1$ will also be denoted by $L$.

Take $A_1$, $A_2$ hyperbolic matrices in $\SL(2,\R)$
such that $\tr A_1$, $\tr A_2 > 2$ and $\tr A_1 A_2 < -2$;
then by \cite[Prop.~3.4]{ABY} there exists a cyclical order $<$ on $\P^1$ such that
$$
u_{A_2} <u_{A_2 A_1} <s_{A_2 A_1} < s_{A_1} < u_{A_1} < u_{A_1 A_2} < s_{A_1 A_2} < s_{A_2} < u_{A_2} \, .
$$
Now take a hyperbolic matrix $A_3 \in \SL(2,\R)$
such that (see Fig.~\ref{f:123}):
$$
u_{A_3} \in (s_{A_1},u_{A_1}), \quad
s_{A_3} \in (s_{A_2},u_{A_2}), \quad \text{and} \quad
A_3 u_{A_2} = s_{A_1} \, .
$$

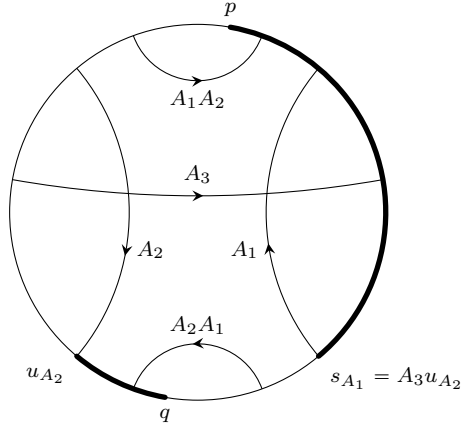
\begin{figure}[hbt]
	\begin{tikzpicture}[scale = 2.5]
		
		
		\draw (0,0) circle [radius=1];
			
		\draw[line width=2pt, line cap=round] (cos 50, -sin 50)
		node[below right]{\footnotesize $s_{A_1} = A_3 u_{A_2}$}
		arc[radius=1, start angle=-50, end angle=80]
		node[above]{\footnotesize $p$};

		\draw[line width=2pt, line cap=round] (cos 230, sin 230)
		node[below left]{\footnotesize $u_{A_2}$}
		arc[radius=1, start angle=230, end angle=260]
		node[below]{\footnotesize $q$};

		\draw[postaction={thick,decorate,decoration={markings, mark=at position .4 with {\arrow{stealth}}}}]  (cos 310, sin 310) arc[radius=tan 50, start angle=220, end angle=140];

		\draw[postaction={thick,decorate,decoration={markings, mark=at position .64 with {\arrow{stealth}}}}]  (cos 130, sin 130) arc[radius=tan 50, start angle=40, end angle=-40];

		\draw[postaction={thick,decorate,decoration={markings, mark=at position .515 with {\arrow{stealth}}}}]  (cos 170, sin 170) arc[radius=tan 80, start angle=260, end angle=280];

		\draw[postaction={thick,decorate,decoration={markings, mark=at position .53 with {\arrow{stealth}}}}] (cos 110, sin 110) arc[radius=tan 20, start angle=200, end angle=340];
		
		\draw[postaction={thick,decorate,decoration={markings, mark=at position .53 with {\arrow{stealth}}}}]  (cos 290, sin 290) arc[radius=tan 20, start angle=20, end angle=160];
	
		\draw (0,.6) node{\footnotesize $A_1 A_2$};
		\draw (0,.2) node{\footnotesize $A_3$};
		\draw (0,-.6) node{\footnotesize $A_2 A_1$};
		\draw (.25,-.2) node{\footnotesize $A_1$};
		\draw (-.25,-.2) node{\footnotesize $A_2$};
		
	\end{tikzpicture}
	\caption{The example of \autoref{p:123}. The thick part represents the (non-strictly) forward-invariant multicone $M$. For each $L$, the arrow labelled $L$ represents the hyperbolic geodesic from $s_L$ to $u_L$.}\label{f:123}
\end{figure}

\begin{proposition}\label{p:123}
The one-step cocycle generated by $\sA\coloneqq(A_1,A_2,A_3)$ has no $\lambda_1$-minimizing measure.
\end{proposition}

We remark that the example is on the boundary of the hyperbolic
component $H \subset \SL(2,\R)^3$ described in \cite[Prop.~4.16]{ABY}.

Before proving the proposition, let us describe a general geometrical construction.
Consider a cocycle given by
$T \colon \Omega \to \Omega$ and $A \colon \Omega \to \GL(2,\R)$.
Let $S$ be the skew-product map on $\Omega \times \P^1$ induced by the cocycle.
The derivative along the $\P^1$ fiber
of the map $S$ at a point $(\omega,x) \in \Omega \times \P^1$
is a linear map
\begin{equation}\label{eq:derivative}
L(\omega,x) \colon T_x \P^1 \to T_{A(\omega) x} \P^1 \, .
\end{equation}
Fix a (unique modulo rescaling) rotation-invariant Riemannian metric on $\P^1$,
and let
$f(\omega, x)$ denote the operator norm of $L(\omega,x)$.
(In other words, $f$ measures the rates of expansion of angle.)

Now suppose that $\mu$ is an ergodic $T$-invariant measure
and $\hat{\mu}$ is a $S$-invariant probability measure that projects to $\mu$.
Then we have the following fact (whose easy proof is left to the reader):

\begin{lemma}\label{l:integral}
If $\lambda_1(A,\mu) = 0$ then
$\int_{\Omega \times \P^1} \log f \, d\hat{\mu} = 0$.
\end{lemma}

\begin{proof}[Proof of \autoref{p:123}]
Let $A_1$, $A_2$, $A_3$ be as above,
and consider the one-step cocycle $(T,A)$,
where $T$ is the shift on $\Omega \coloneqq \{1,2,3\}^{\Z_+}$,
and $A \colon \Omega \to \SL(2,\R)$ is given by $A(\omega) = A_{\omega_0}$.
Let $S$ be the induced skew-product map on $\Omega \times \P^1$.

Due to the ``heteroclinic connection'' $A_3 u_{A_2} = s_{A_1}$,
the cocycle is not uniformly hyperbolic and so, as explained in \autoref{ss:alternative},
we have $\lambda_1^\bot(\sA) = 0$.
To prove the proposition we will show that
$\lambda_1(\sA,\mu) > 0$
for every ergodic $\mu \in \cM_T$.

Fix a point $q$ in the interval $(u_{A_2 A_1},s_{A_2 A_1})$.
Note that 
$$
u_{A_1 A_2} < A_1 q < A_2^{-1} q < s_{A_1 A_2} \, .
$$
In particular, we can fix a point $p$ in the interval $(A_1 q, A_2^{-1} q)$.
Let $M\coloneqq (s_{A_1},p) \cup (u_{A_2},q)$, as in Fig.~\ref{f:123}.
Then the set $M$ is forward-invariant under the projection action of
each matrix $A_i$; for example the case $i=2$ follows from the fact that $A_2 p \in (u_{A_2},q)$.

Endow each connected component of $M$
with its Riemannian Hilbert metric.
Given a point $(\omega,x) \in \Omega \times M$,
let $g(\omega,x)$ denote the operator norm of the linear map \eqref{eq:derivative},
where we take Hilbert metrics on both tangent spaces.
Since the set $A(\omega)(M)$ is contained in $M$
and none of its connected components coincided with a connected component of $M$,
we have $g(\omega,x) < 1$.

Take any ergodic $T$-invariant measure $\mu$,
and lift it to a $S$-invariant measure $\hat{\mu}$
supported on the forward $S$-invariant compact set $\Omega \times \overline{M}$.
We can assume that $\mu$ is neither $\delta_{1^\infty}$ nor $\delta_{2^\infty}$,
because otherwise $\lambda_1(\sA,\mu) >0$ trivially.
It is then easy to see that $\hat{\mu}$ gives zero weight to
the subset $\Omega \times \partial M$,
and in particular the integral $I \coloneqq \int \log g \, d\hat{\mu}$ is well-defined.
It is immediate from the definitions that $\log g - \log f$
is a coboundary with respect to $S$, and therefore $\int \log f \, d\hat{\mu} = I$.
Since $g<0$, we have $I<0$
and so \autoref{l:integral} gives
$\lambda_1(\sA,\mu) \neq 0$, as we wanted to show.
\end{proof}

\subsection{Examples of non-uniqueness of optimizing measures} \label{ss:non_uniqueness}

Let us show that in the context of \autoref{t:zero},
the Mather sets $K^\top$ and $K^\bot$ are not necessarily uniquely ergodic.
In other words, the $\lambda_1$-maximizing and $\lambda_1$-minimizing measures can fail to be unique.

Take a pair of matrices $A_1$ and $A_2$ in $\GL(2,\R)^2$
with respective eigenvalues $\chi_1(A_1) > \chi_2(A_1)$ and $\chi_1(A_2) > \chi_2(A_2)$,
all of them positive.
Let $v_j(A_i) \in \P^1$ be the eigendirection of $A_i$ corresponding to the eigenvalue $\chi_j(A_i)$.
We can choose the pair $\sA = (A_1,A_2)$ so that:
\begin{itemize}
	\item the geodesics $\overrightarrow{v_2(A_1) v_1(A_1)}$ and $\overrightarrow{v_2(A_2) v_1(A_2)}$ cross;
	\item $\sA$ has a strictly forward-invariant cone $M \subset \P^1$ with the forward nonoverlapping property;
	\item $\sA$ has a strictly backwards-invariant cone $N \subset \P^1$ with the backards nonoverlapping property.
\end{itemize}
See \autoref{f:two_tops}.

\begin{figure}[htb]
	\centering
	\begin{minipage}{0.4\textwidth}
		\centering
		\begin{tikzpicture}[scale = 2]
			\draw (0,0) circle [radius=1];

			\draw[line width=2pt, line cap=round] (cos 297, sin 297)
			arc[radius=1, start angle=297, end angle=345];
			\draw (cos 35, sin 35) node[above right]{\footnotesize $A_1(M)$};

			\draw[line width=2pt, line cap=round] (cos 15, sin 15)
			arc[radius=1, start angle=15, end angle=63];
			\draw (cos 325, sin 325) node[below right]{\footnotesize $A_2(M)$};

			\draw[line width=2pt, line cap=round] (cos 117, sin 117)
			arc[radius=1, start angle=117, end angle=165];
			\draw (cos 145, sin 145) node[above left]{\footnotesize $A_2^{-1}(N)$};

			\draw[line width=2pt, line cap=round] (cos 195, sin 195)
			arc[radius=1, start angle=195, end angle=243];
			\draw (cos 215, sin 215) node[below left]{\footnotesize $A_1^{-1}(N)$};

			\draw[postaction={thick,decorate,decoration={markings, mark=at position .75 with {\arrow{stealth}}}}] (cos 240, sin 240) node[below left]{\footnotesize $v_2(A_1)$} -- (cos 60, sin 60) node[above right]{\footnotesize $v_1(A_1)$};
			
			\draw[postaction={thick,decorate,decoration={markings, mark=at position .75 with {\arrow{stealth}}}}] (cos 120, sin 120) node[above left]{\footnotesize $v_2(A_2)$} -- (cos 300, sin 300) node[below right]{\footnotesize $v_1(A_2)$};

			\draw[postaction={thick,decorate,decoration={markings, mark=at position .53 with {\arrow{stealth}}}}] (cos 160, sin 160) node[left]{\footnotesize $e_2(\xi)$} arc[radius=tan 70, start angle=250, end angle=290] node[right]{\footnotesize $e_1(\xi)$};

			\draw[postaction={thick,decorate,decoration={markings, mark=at position .53 with {\arrow{stealth}}}}] (cos 200, sin 200) node[left]{\footnotesize $e_2(\eta)$} arc[radius=tan 70, start angle=110, end angle=70] node[right]{\footnotesize $e_1(\eta)$};

		\end{tikzpicture}
		\caption{An example with $K^\top = \{1^\infty, 2^\infty\}$.}
		\label{f:two_tops}
	\end{minipage}
	\hfill
	\begin{minipage}{0.4\textwidth}
		\centering
		\begin{tikzpicture}[scale = 2]
			\draw (0,0) circle [radius=1];
			
			\draw[line width=2pt, line cap=round] (cos 297, sin 297)
			arc[radius=1, start angle=297, end angle=345];
			\draw (cos 35, sin 35) node[above right]{\footnotesize $A_1(M)$};

			\draw[line width=2pt, line cap=round] (cos 15, sin 15)
			arc[radius=1, start angle=15, end angle=63];
			\draw (cos 325, sin 325) node[below right]{\footnotesize $A_2(M)$};

			\draw[line width=2pt, line cap=round] (cos 117, sin 117)
			arc[radius=1, start angle=117, end angle=165];
			\draw (cos 145, sin 145) node[above left]{\footnotesize $A_1^{-1}(N)$};

			\draw[line width=2pt, line cap=round] (cos 195, sin 195)
			arc[radius=1, start angle=195, end angle=243];
			\draw (cos 215, sin 215) node[below left]{\footnotesize $A_2^{-1}(N)$};

			\draw[postaction={thick,decorate,decoration={markings, mark=at position .53 with {\arrow{stealth}}}}]  (cos 120, sin 120) node[above left]{\footnotesize $v_2(A_1)$} arc[radius=tan 30, start angle=210, end angle=330] node[above right]{\footnotesize $v_1(A_1)$};
			
			\draw[postaction={thick,decorate,decoration={markings, mark=at position .53 with {\arrow{stealth}}}}] (cos 240, sin 240) node[below left]{\footnotesize $v_2(A_2)$} arc[radius=tan 30, start angle=150, 						
end angle=30] node[below right]{\footnotesize $v_1(A_2)$};
			
			\draw[postaction={thick,decorate,decoration={markings, mark=at position .75 with {\arrow{stealth}}}}] (cos 160, sin 160) node[left]{\footnotesize $e_2(\eta)$} -- (cos 340, sin 340) node[right]{\footnotesize $e_1(\eta)$};
			
			\draw[postaction={thick,decorate,decoration={markings, mark=at position .75 with {\arrow{stealth}}}}] (cos 200, sin 200) node[left]{\footnotesize $e_2(\xi)$} -- (cos 20, sin 20) node[right]{\footnotesize $e_1(\xi)$};
			
		\end{tikzpicture}
		\caption{An example with $K^\bot = \{1^\infty, 2^\infty\}$.}
		\label{f:two_bots}
	\end{minipage}
\end{figure}

\begin{claim}\label{cl:exclusion}
If $\xi$, $\eta \in K^\top$ are such that
\begin{equation}\label{eq:two_transitions}
\xi_{-1} = 1, \quad \xi_{0} = 2, \quad \eta_{-1} = 2, \quad \eta_{0} = 1.
\end{equation}
then $\xi \not\in K^\top$ or  $\eta \not\in K^\top$.
\end{claim}

\begin{proof}[Proof of the claim]
The four relations in \eqref{eq:two_transitions} respectively imply:
$$
e_1(\xi) \in A_1(M), \quad e_2(\xi) \in A_2^{-1}(N), \quad e_1(\eta) \in A_2(M), \quad e_2(\eta) \in A_1^{-1}(N).
$$
It follows that the geodesics
$\overrightarrow{e_2(\xi)e_1(\xi)}$ and $\overrightarrow{e_2(\eta)e_1(\eta)}$ are coparallel
(see \autoref{f:two_tops}).
The claim now follows from \autoref{c:forbid}.
\end{proof}

Let $1^\infty$ and $2^\infty \in \{1,2\}^\Z$ be the two fixed points of the shift,
and let $\zeta^{12}$ and $\zeta^{21} \in k^\Z$ be the following ``homoclinic points'':
$$
\zeta^{12}_n =
\begin{cases}
1 &\quad\text{if $n<0$,}\\
2 &\quad\text{if $n \ge 0$,}
\end{cases}
\qquad
\zeta^{21}_n =
\begin{cases}
2 &\quad\text{if $n<0$,}\\
1 &\quad\text{if $n \ge 0$.}
\end{cases}
$$
It follows from \autoref{cl:exclusion} that
$K^\top$ is contained in the closure of the orbit of either $\zeta^{12}$ or $\zeta^{21}$.
Since $K^\top$ equals the union of supports of the invariant probability measures
that give full weight to $K^\top$ itself,
it follows that $K^\top \subset \{1^\infty, 2^\infty\}$.

Of course we can choose $A_1$, $A_2$ such that additionally
$\chi_1(A_1) = \chi_1(A_2)$; in this case $K^\top$ equals $\{1^\infty, 2^\infty\}$
and so it is not uniquely ergodic.

In a very similar way we produce an example where $K^\bot = \{1^\infty, 2^\infty\}$.
The only difference is that $\sA = (A_1,A_2)$ are chosen so that
the geodesics $\overrightarrow{v_2(A_1) v_1(A_1)}$ and $\overrightarrow{v_2(A_2) v_1(A_2)}$ are coparallel,
and so if the points $\xi$, $\eta$ satisfy \eqref{eq:two_transitions}
then the geodesics $\overrightarrow{e_2(\xi)e_1(\xi)}$ and $\overrightarrow{e_2(\eta)e_1(\eta)}$ cross.
(See \autoref{f:two_bots}.)

\subsection{Open questions and directions for future research}\label{ss:open}

There are several different directions along which one could try to extend the results of this paper.

Notice that the NOC is indeed necessary for the validity of \autoref{t:zero};
an example is given in \autoref{r:NOCs} for $\alpha=\beta$.
However all the examples we know are very non-generic.
So we ask whether the NOC can be replaced by a weaker condition,
preferably one that is ``typical'' (open and dense) among $k$-tuples of matrices that generate
dominated cocycles.

\medskip

Regarding more general cocycles,
we remark that there is also a notion of multicones for one-step cocycles over
subshifts of finite type: see \cite{ABY}.
It seems to be straightforward to adapt the arguments given here
to that more general situation
(and thus also for $n$-step cocycles) with appropriate nonoverlapping conditions,
but we have not checked the details.

Even more generally, we would like to have results about Lyapunov-optimizing measures
for cocycles that are not locally constant.
We believe that some of the construction of this paper should extend
to cocycles admitting unstable and stable holonomies (over a hyperbolic base dynamics).

\medskip

Let us return to one-step cocycles over the full shift.
A possible strengthening of the conclusions of \autoref{t:zero}
would be to replace subexponential complexity (zero entropy)
by linear complexity (as in \cite{BM,JP}),
or polynomial complexity (as in \cite{HMS}).
Perhaps under generic conditions we can even
obtain bounded complexity (periodic orbits), in the style of \cite{Contreras}.

\medskip

Another line of study is to consider a relative Lyapunov-optimization problem for one-step cocycles
where the frequencies of each matrix are fixed.
The paper \cite{JS90} deals with a problem which can be reformulated in this terms.
See \cite{GaLo} for general results on relative optimization in the classical commutative setting.
Let us also remark that this relative optimization setting is natural in the context of Lagrangian dynamics, where it corresponds to fixing the homology; see \cite{Mather}.

It should also be worthwhile to investigate the relations between Lyapunov-optimizing results
as ours and the geometry of Riemann surfaces.

\medskip

Regarding non-dominated one-step  $\SL(2,\R)$-cocycles,
\autoref{t:positive} says that we should not expect $\lambda_1$-minimizing measures
to have zero entropy.
However, it seems likely that $\lambda_1$-maximizing measures should have zero entropy.
Notice that the corresponding Mather set (whose existence is given by \cite{Morris13})
is automatically uniformly hyperbolic.

Let us also remark that the only examples of $k$-tuples of matrices
that do not satisfy the dichotomy of \autoref{t:positive} (or \autoref{c:positive})
are very particular ones (e.g., appropriate $k$-tuples with a common invariant direction).
So we ask whether these counterexamples can be described explicitly,
or at least whether they are contained in a finite union of submanifolds of positive codimension.

\medskip

Of course most of the concepts and questions discussed in this paper make sense in higher dimension.
In particular, we ask whether a higher-dimensional version of our zero entropy \autoref{t:zero}
(stated in terms of domination of index $1$) holds true.
As mentioned above, the construction of Barabanov functions can be adapted to this situation:
see \cite[\S~2.2]{BMo}.
\autoref{l:ineq_cross_ratio_J} should also be possible to extend:
compare with \cite[Prop.~2.6]{BM}.
However, the rest of our proof relies on low-dimensional arguments.

\medskip

Finally, we remark that the results obtained here can be considered as
part of the multifractal analysis of Lyapunov exponents of linear cocycles,
a broad field of study launched essentially by Feng~\cite{Feng}.

\bigskip
\begin{ack}
We thank Artur Avila, Gonzalo Contreras, Oliver Jenkinson, Artur O.~Lopes,
Ian D.~Morris, Eduardo Garibaldi, Mark Pollicott, and Philippe Thieullen
for enlightening discussions.
We also thank two referees for their careful reading, and important corrections and suggestions.
The first named author acknowledges the hospitality of IMPAN and the support of CNPq (Brazil), FAPERJ (Brazil), Fondecyt 1140202 (Chile), and the Center of Dynamical Systems and Related Fields (Chile).
The second named author was hosted by PUC-Rio and supported by MNiSW grant
N201~607640 (Poland).
Both authors were also supported by EU~BREUDS.
\end{ack}


\end{document}